\documentclass[a4paper,11pt]{article}
\usepackage[OT1]{fontenc}
\usepackage{amsthm,amsmath,graphicx,latexsym,amssymb,array}
\usepackage{natbib}
\usepackage[colorlinks,citecolor=blue,urlcolor=blue]{hyperref}

\numberwithin{equation}{section}
\theoremstyle{plain}
\newtheorem{theorem}{Theorem}[section]

\newtheorem{definition}{Definition}[section]

\newtheorem{remark}{Remark}[section]

\title{Tests for high dimensional data based on means,\\spatial signs and spatial ranks}
\date{}
\author{\vspace{0.3in} Anirvan Chakraborty and Probal Chaudhuri}
\begin{document}

\maketitle
\vspace{-0.6in}
\begin{center}
Theoretical Statistics and Mathematics Unit, \\
Indian Statistical Institute \\
203, B. T. Road, Kolkata - 700108, INDIA. \\
emails: vanchak@gmail.com, probal@isical.ac.in
\end{center}
\vspace{0.15in}
\begin{abstract}
Tests based on sample mean vectors and sample spatial signs have been studied in the recent literature for high dimensional data with the dimension larger than the sample size. For suitable sequences of alternatives, we show that the powers of the mean based tests and the tests based on spatial signs and ranks tend to be same as the data dimension grows to infinity for any sample size, when the coordinate variables satisfy appropriate mixing conditions. Further, their limiting powers do not depend on the heaviness of the tails of the distributions. This is in striking contrast to the asymptotic results obtained in the classical multivariate setup. On the other hand, we show that in the presence of stronger dependence among the coordinate variables, the spatial sign and rank based tests for high dimensional data can be asymptotically more powerful than the mean based tests if in addition to the data dimension, the sample size also grows to infinity. The sizes of some mean based tests for high dimensional data studied in the recent literature are observed to be significantly different from their nominal levels. This is due to the inadequacy of the asymptotic approximations used for the distributions of those test statistics. However, our asymptotic approximations for the tests based on spatial signs and ranks are observed to work well when the tests are applied on a variety of simulated and real datasets.
\vspace{0.1in} 
\end{abstract}

\section{Introduction}
\label{1}
\indent For univariate data, nonparametric tests based on signs and ranks are well-known competitors of tests based on sample means like the $t$-test. These nonparametric tests have distribution-free property, and they are asymptotically more efficient than the mean based tests for non-Gaussian distributions having heavy tails. Although various extensions of these nonparametric tests have been proposed for multivariate data (see \cite{PS71}, \cite{Oja10} and \cite{HM11}), they do not have the distribution-free property in general, and they are often implemented using their permutation distributions. However, like their univariate counterparts, they are usually asymptotically more efficient than the mean based Hotelling's $T^{2}$ test for multivariate non-Gaussian distributions with heavy tails (see \cite{CM97}, \cite{MOT97}, \cite{Mard99} and \cite{Oja10}). \\
\indent For high dimensional data, where the data dimension is larger than the sample size, Hotelling's $T^{2}$ test is not applicable due to the singularity of the sample dispersion matrix. Let ${\bf X}_{1},{\bf X}_{2},\ldots,{\bf X}_{m}$ and ${\bf Y}_{1},{\bf Y}_{2},\ldots,{\bf Y}_{n}$ be i.i.d. copies of independent random vectors ${\bf X}$ and ${\bf Y}$ in $\mathbb{R}^{d}$. For testing $H_{0} : E({\bf X}) = E({\bf Y})$ against the alternative $H_{A} : E({\bf X}) \neq E({\bf Y})$ for two high dimensional observations ${\bf X}$ and ${\bf Y}$, \cite{BS96} proposed a test based on $||\overline{{\bf X}} - \overline{{\bf Y}}||^{2}$, where $\overline{{\bf X}}$ and $\overline{{\bf Y}}$ are the sample means of the two samples. \cite{CQ10} proposed a test statistic after removing the terms $\sum_{i=1}^{m} ||{\bf X}_{i}||^{2}$ and $\sum_{j=1}^{n} ||{\bf Y}_{j}||^{2}$ appearing in the expansion of $||\overline{{\bf X}} - \overline{{\bf Y}}||^{2}$, which makes the resulting statistic an unbiased estimator of $||E({\bf X} - {\bf Y})||^{2}$. The one sample and the two sample statistics of \cite{CQ10} based on sample means are
\begin{eqnarray*}
T_{CQ}^{(1)} = \frac{1}{(n)_{2}} \sum_{\substack{i_{1}, i_{2} = 1, \\i_{1} \neq i_{2}}}^{m} {\bf X}_{i_{1}}'{\bf X}_{i_{2}}, \ \ \ \mbox{and} \\
T_{CQ}^{(2)} = \frac{1}{(m)_{2}(n)_{2}} \sum_{\substack{i_{1}, i_{2} = 1, \\i_{1} \neq i_{2}}}^{m} \sum_{\substack{j_{1}, j_{2} = 1, \\j_{1} \neq j_{2}}}^{n} ({\bf X}_{i_{1}} - {\bf Y}_{j_{1}})'({\bf X}_{i_{2}} - {\bf Y}_{j_{2}}),
\end{eqnarray*}
respectively, where $(p)_{q} = p(p-1)\ldots(p-q+1)$ for integers $p \geq 1$ and $1 \leq q < p$. \\
\indent Well known multivariate spatial sign and rank based tests (see \cite{MO95}, \cite{MOT97}, \cite{CM97}, \cite{Mard99} and \cite{Oja10}) also involve inverses of dispersion matrices computed from the sample, which become singular when the data dimension exceeds the sample size. \cite{WPL15} proposed a one sample test of the mean vector based on spatial signs given by
$$T_{S} = \frac{1}{(n)_{2}} \sum_{\substack{i_{1}, i_{2} = 1, \\i_{1} \neq i_{2}}}^{n} S({\bf X}_{i_{1}})'S({\bf X}_{i_{2}}),$$
where $S({\bf x}) = {\bf x}/||{\bf x}||$ denotes the spatial sign of any ${\bf x} \in \mathbb{R}^{d}$. A natural high dimensional version of the one sample spatial signed rank statistic can be defined using the idea of \cite{WPL15}, and it is given by
\begin{eqnarray*}
T_{SR} = \frac{1}{(n)_{4}} \sum_{\substack{i_{1}, i_{2}, i_{3}, i_{4}\\all~distinct}} S({\bf Z}_{i_{1}} + {\bf Z}_{i_{2}})'S({\bf Z}_{i_{3}} + {\bf Z}_{i_{4}}).
\end{eqnarray*}
Similarly, a two sample spatial rank statistic can be defined as
\begin{eqnarray*}
T_{WMW} = \frac{1}{(m)_{2}(n)_{2}} \sum_{\substack{i_{1}, i_{2} = 1, \\i_{1} \neq i_{2}}}^{m} \sum_{\substack{j_{1}, j_{2} = 1, \\j_{1} \neq j_{2}}}^{n} S({\bf Y}_{j_{1}} - {\bf X}_{i_{1}})'S({\bf Y}_{j_{2}} - {\bf X}_{i_{2}}).
\end{eqnarray*}
Note that $T_{S}$, $T_{SR}$ and $T_{WMW}$ are unbiased estimators of $||E\{S({\bf X}_{1})\}||^{2}$, $||E\{S({\bf X}_{1} + {\bf X}_{2})\}||^{2}$ and $||E\{S({\bf X} - {\bf Y})\}||^{2}$, respectively. \\
\indent In this article, we study the behaviours of different tests based on sample means, spatial signs and ranks under various probability models for high dimensional data. In Section \ref{2}, we prove that under appropriate mixing conditions on the coordinate variables and suitable sequences of alternatives, the limiting powers of the spatial rank based test and the mean based tests are the same as the data dimension grows to infinity. This is true for all sample sizes and irrespective of the heaviness of the tails of the underlying distributions. Analogous results hold for the one sample spatial sign and signed rank based tests and the mean based tests, and those are presented in subsection \ref{2.1}. These results are in striking contrast to the asymptotic results obtained in the traditional multivariate setup, where the data dimension is fixed and the sample sizes grow to infinity. In such a setup, the multivariate spatial sign and rank based tests are asymptotically less efficient than Hotelling's $T^{2}$ test for Gaussian distributions, and they are more efficient than the $T^{2}$ test for non-Gaussian distributions with heavy tails (see \cite{MOT97}, \cite{CM97}, \cite{Mard99} and \cite{Oja10}). Recall that for multivariate Gaussian data, the Hotelling's $T^{2}$ test is actually the likelihood ratio test and the most powerful invariant test. In Section \ref{3}, we prove that in the presence of some stronger dependence among the coordinate variables, the limiting powers of the spatial sign and rank based tests can be more than those of their competitors based on sample means if we first let the data dimension and then the sample size to grow to infinity. In Section \ref{4}, we demonstrate the performances of the tests based on sample means and spatial signs and ranks using some real datasets. In Section \ref{5}, we discuss the performances of these tests in comparison with some other mean based tests for high dimensional data available in recent literature. It is found that the sizes of some of the mean based tests are significantly different from their nominal sizes due to the inadequacy of the asymptotic approximations used for the distributions of the corresponding test statistics. The proofs of all the theorems are presented in Appendix -- I.

\section{Asymptotic behaviours of different tests under $\rho$-mixing}
\label{2}
\indent Let ${\cal X} = (X_{1},X_{2},\ldots)$ be an infinite sequence of random variables defined over a probability space $(\Omega,{\cal A},P)$.
\begin{definition}[\cite{KR60}] \label{def1}
A sequence ${\cal X}$ is said to be $\rho$-mixing if $\rho(d) = \sup_{k \geq 1} \sup_{f \in {\cal F}_{k}, g \in {\cal F}_{d+k}} |Corr(f,g)|$ converges to zero as $d \rightarrow \infty$. Here, $\rho(\cdot)$ is called the $\rho$-mixing coefficient of ${\cal X}$, and ${\cal F}_{k}$ denotes the $\sigma$-field generated by measurable square integrable functions of $(X_{1},X_{2},\ldots,X_{k})$ for $k \geq 1$.
\end{definition}
We refer to \cite{LL96} and \cite{Brad05} for further details about $\rho$-mixing sequences. Let ${\bf X}_{1},{\bf X}_{2},\ldots,{\bf X}_{m}$ and ${\bf Y}_{1},{\bf Y}_{2},\ldots,{\bf Y}_{n}$ be i.i.d. copies of independent random vectors ${\bf X}$ and ${\bf Y}$ in $\mathbb{R}^{d}$. We assume the following conditions.
\vspace{0.02in}\\
{\it (C1) ${\bf X} = \mu_{1} + {\bf V}$ and ${\bf Y} = \mu_{2} + {\bf W}$ for some $\mu_{1}, \mu_{2} \in \mathbb{R}^{d}$, where ${\bf V}$ and ${\bf W}$ are vectors formed by the first $d$ coordinates of the zero mean, strictly stationary, and $\rho$-mixing sequences ${\cal V} = (V_{1},V_{2},\ldots)$ and ${\cal W} = (W_{1},W_{2},\ldots)$ satisfying $E(V_{1}^{4}) < \infty$ and $E(W_{1}^{4}) < \infty$.  \\
(C2) The $\rho$-mixing coefficients $\rho_{1}(\cdot)$ and $\rho_{2}(\cdot)$ of ${\cal V}$ and ${\cal W}$ satisfy $\sum_{k=1}^{\infty} \rho_{1}(2^{k}) < \infty$ and $\sum_{k=1}^{\infty} \rho_{2}(2^{k})$ $< \infty$, respectively.}  \\
Denote $\mu = \mu_{2} - \mu_{1}$, $\sigma_{1}^{2} = Var(X_{1}) > 0$, $\sigma_{2}^{2} = Var(Y_{1}) > 0$, $\Sigma_{1} = Disp({\bf X})$, and $\Sigma_{2} = Disp({\bf Y})$, where ${\bf X} = (X_{1},X_{2},\ldots,X_{d})$ and ${\bf Y} = (Y_{1},Y_{2},\ldots,Y_{d})$. \\
{\it (C3) $||\mu||^{2}/d^{1/2+\epsilon} \rightarrow 0$ for some $\epsilon > 0$ and $\mu'(\Sigma_{1} + \Sigma_{2})\mu = o({\mbox{tr}}(\Sigma_{1}^{2} + \Sigma_{2}^{2}))$ as $d \rightarrow \infty$. }
\vspace{0.02in}\\
Examples of $\rho$-mixing sequences include $m$-dependent sequences, stationary ARMA($p$,$q$) processes with white noise innovation process (see \citet[Theorem 1.1.2]{LL96}), and hidden Markov models whose underlying generator sequences are stationary, Gaussian and geometrically ergodic Markov chains (see \citet[Theorem 3.7]{Brad05}). For all of the above models, condition (C2) holds. Condition (C3) is trivially true under the null hypothesis $H_{0} : \mu = {\bf 0}$. Note that when $\Sigma_{1}$ and $\Sigma_{2}$ are identity matrices, the second part of condition (C3) is automatically true if its first part holds. In general, the second part of condition (C3) holds if in addition to the first part, we have $\lambda_{d}^{-1} \sum_{k=1}^{d} \lambda_{k}^{2} = O(d^{1/2 + \epsilon})$ as $d \rightarrow \infty$, where $\lambda_{1} < \lambda_{2} < \ldots < \lambda_{d}$ are the eigenvalues of $\Sigma_{1} + \Sigma_{2}$.\\
\indent \cite{CQ10} worked in a setup, where ${\bf X}$ and ${\bf Y}$ are affine transformations of certain zero mean random vectors, whose coordinates are ``pseudo-independent'' (see (3.2) in p. 811 in that paper). The distributional assumptions in (C1) and (C2) cover many distributions that satisfy the model assumptions stated in (3.1) in \citet[p. 811]{CQ10}, e.g., distributions with independent coordinates, moving average processes and more generally m-dependent sequences as well as autoregressive processes. \cite{FL98} considered the problem of testing equality of two mean curves for functional data, and they modelled the data as a finite dimensional one, where the data dimension is larger than the sample size. A class of probability models considered by them are stationary linear Gaussian processes, many of which satisfy the model assumptions considered above. \cite{SKK13} studied a two sample mean based test based on the sum of squares of the coordinatewise $t$ statistics and studied its properties assuming multivariate Gaussianity of the data, which includes many distributions satisfying Assumptions (C1) and (C2). A closely related test was proposed by \cite{GCBL14}, and they studied its properties under $\alpha$-mixing (see \cite{LL96}) conditions on the data, which is weaker than the $\rho$-mixing setup considered above. However, those authors required the existence of sixteenth order moments. \cite{CLX14} proposed a mean based test for detecting sparse alternatives and studied its properties primarily under the assumption of multivariate Gaussianity of the data. \cite{FZWZ15} proposed a modification of the test in \cite{SKK13} and they worked in a setup similar to that considered by \cite{CQ10}. Thus, as in the case of the latter paper, many probability distributions included in the setup considered by \cite{FZWZ15} satisfy the $\rho$-mixing assumptions described here. \cite{WLWM15} studied the properties of their test under spherical Gaussian distributions, which are special cases of the $\rho$-mixing models considered here.
\begin{theorem}  \label{thm-new000}
Suppose that conditions (C1)--(C3) are satisfied. Define, $\Gamma_{1} = 2{\mbox{tr}}(\Sigma_{1}^{2})/(m)_{2} + 2{\mbox{tr}}(\Sigma_{2}^{2})/(n)_{2} + 4{\mbox{tr}}(\Sigma_{1}\Sigma_{2})/(mn)$. Then, each of $[d(\sigma_{1}^{2} + \sigma_{2}^{2})T_{WMW} - ||\mu||^{2}]/{\Gamma_{1}}^{1/2}$ and $(T_{CQ}^{(2)} - ||\mu||^{2})/{\Gamma_{1}}^{1/2}$ converges weakly to a standard Gaussian variable as $d \rightarrow \infty$ for every fixed $m,n \geq 1$.
\end{theorem}
\indent When the null hypothesis $H_{0} : \mu = {\bf 0}$ is true, the above theorem yields the asymptotic null distributions of $T_{WMW}$ and $T_{CQ}^{(2)}$ as $d \rightarrow \infty$. Let us observe that the asymptotic distribution of $T_{CQ}^{(2)}$ obtained in the above theorem as $d \rightarrow \infty$ is the same as that obtained by \cite{CQ10} in their Theorem 1 when both $d, n \rightarrow \infty$. These authors used an assumption similar to that in the second part of condition (C3) for deriving the asymptotic distribution of their test statistic, when both $d$ and $n$ are large (see (3.4) in p. 812 in \cite{CQ10}). \\
\indent When the alternative hypothesis $H_{A} : \mu \neq {\bf 0}$ is true, the next theorem compares the asymptotic powers of the tests based on $T_{WMW}$ and $T_{CQ}^{(2)}$ for high dimensional data. Let $\beta_{T_{WMW}}(\mu)$ and $\beta_{T_{CQ}^{(2)}}(\mu)$ be the powers of these two tests at a given level of significance.
\begin{theorem}  \label{thm-new0}
Suppose that conditions (C1)--(C3) are satisfied, and assume $\lim_{d \rightarrow \infty} ||\mu||^{2}/\Gamma_{1}^{1/2} = c$ for some $c \in [0,\infty]$. Then, $\lim_{d \rightarrow \infty} \beta_{T_{WMW}}(\mu) = \lim_{d \rightarrow \infty} \beta_{T_{CQ}^{(2)}}(\mu) = \beta$ for every fixed $m,n \geq 1$, where $\beta = \alpha$, $\beta = 1$, or $\beta \in (\alpha,1)$ according as $c = 0$, $c = \infty$, or $c \in (0,\infty)$, respectively. Here, $\alpha$ is the level of significance of the test.
\end{theorem}
\indent The above theorem implies that the asymptotic powers of the mean based and the spatial rank based tests are the same as $d \rightarrow \infty$ for each fixed $m,n \geq 1$. If $\Sigma_{1}$ and $\Sigma_{2}$ equal the $d \times d$ identity matrix, and $d$ is large, we get different powers of the tests based on $T_{WMW}$ and $T_{CQ}^{(2)}$ according as $||\mu||/d^{1/4}$ converges to zero, infinity or some $c \in (0,\infty)$.

\subsection{Empirical study using some $\rho$-mixing models}
\label{2.1}
\indent For implementing the tests based on $T_{WMW}$ and $T_{CQ}^{(2)}$ under the $\rho$-mixing setup, we can use their limiting null distributions obtained from Theorem \ref{thm-new000} after plugging-in the following unbiased estimators of the parameters involved.
\begin{eqnarray*}
&& \widehat{\Gamma_{1}} = \frac{2}{(m)_{2}}\widehat{{\rm{tr}}(\Sigma_{1}^{2})} + \frac{2}{(n)_{2}}\widehat{{\rm{tr}}(\Sigma_{2}^{2})} + \frac{4}{mn}\widehat{{\rm{tr}}(\Sigma_{1}\Sigma_{2})}, \\
\mbox{where} && \widehat{{\rm{tr}}(\Sigma_{1}^{2})} = \frac{1}{4(m)_{4}} \sum_{\substack{i_{1}, i_{2}, i_{3}, i_{4} \\ all~distinct}} [({\bf X}_{i_{1}} - {\bf X}_{i_{2}})'({\bf X}_{i_{3}} - {\bf X}_{i_{4}})]^{2}, \\
&& \widehat{{\rm{tr}}(\Sigma_{2}^{2})} = \frac{1}{4(n)_{4}} \sum_{\substack{j_{1}, j_{2}, j_{3}, j_{4} \\ all~distinct}} [({\bf Y}_{j_{1}} - {\bf Y}_{j_{2}})'({\bf Y}_{j_{3}} - {\bf Y}_{j_{4}})]^{2}, \ \ \ \mbox{and} \\
&& \widehat{{\rm{tr}}(\Sigma_{1}\Sigma_{2})} = \frac{1}{4(m)_{2}(n)_{2}} \sum_{i_{1} \neq i_{2}}
\sum_{j_{1} \neq j_{2}} [({\bf X}_{i_{1}} - {\bf X}_{i_{2}})'({\bf Y}_{j_{1}} - {\bf
Y}_{j_{2}})]^{2}, \\
\end{eqnarray*}
Also, $\widehat{\sigma_{1}^{2}} = [d(m-1)]^{-1} \sum_{k=1}^{d} \sum_{i=1}^{m} (X_{ik} - \overline{X}_{k})^{2}$, where $\overline{X}_{k} = d^{-1} \sum_{i=1}^{m} X_{ik}$ with ${\bf X}_{i} = (X_{i1},X_{i2},\ldots,$ $X_{id})$, $1 \leq i \leq m$, and $\widehat{\sigma_{2}^{2}} = [d(n-1)]^{-1} \sum_{k=1}^{d} \sum_{j=1}^{n} (Y_{jk} - \overline{Y}_{k})^{2}$, where $\overline{Y}_{k} = d^{-1} \sum_{j=1}^{n} Y_{jk}$ with ${\bf Y}_{j} = (Y_{j1},Y_{j2},\ldots,Y_{jd})$, $1 \leq j \leq n$. Note that $\widehat{\Gamma}_{1}$ is invariant under location transformations unlike the estimator proposed by \citet[p. 815]{CQ10}. Moreover, for all simulated datasets and real datasets considered later, the empirical sizes and powers of the test based on $T_{CQ}^{(2)}$ implemented as above are similar to those of the original two sample test in \cite{CQ10}. \\
\indent To compare the performances of the tests based on $T_{WMW}$ and $T_{CQ}^{(2)}$, we have considered the $AR(1)$ models with correlation $0.7$ having Gaussian and $t(5)$ innovations. The sample sizes are $m = n = 20$, and $\mu = (c,0,0,\ldots,0)$ with $c = 1.5, 3, 4.5, 6, 7.5$ for $d = 100, 200, 400, 800, 1600$, respectively. The sizes and the powers of the tests based on $T_{WMW}$ and $T_{CQ}^{(2)}$ are averaged over $1000$ Monte Carlo simulations. We found that the sizes of the tests are not significantly different from the nominal $5\%$ level for both the models. It is seen from Figure \ref{Fig1} that the powers of these two tests are similar for all data dimensions considered under both the models. The power curves are so close that they are overlaid on each other.
\begin{figure}[ht!]
\begin{center}
\includegraphics[width=4.5in,height=1.5in]{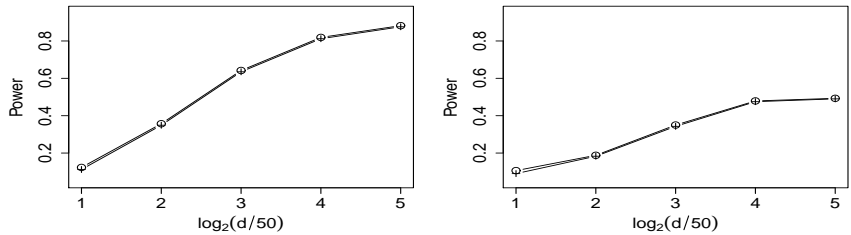}
\end{center}
\vspace{-0.15in}
\caption{Powers of the tests at nominal $5\%$ level based on $T_{WMW}$ (- + - curves) and $T_{CQ}^{(2)}$ (- $\circ$ - curves) for the $AR(1)$ model with Gaussian innovation (left panel) and $t(5)$ innovation (right panel). The two power curves are overlaid on each other in both the plots. \label{Fig1}}
\end{figure}

\subsection{Asymptotic behaviours of one sample tests under $\rho$-mixing}
\label{2.2}
\indent Let ${\bf X}_{1}$, ${\bf X}_{2}, \ldots, {\bf X}_{n}$ be i.i.d. copies of a random vector ${\bf X} \in \mathbb{R}^{d}$. The following theorem gives the asymptotic distributions of $T_{S}$, $T_{SR}$ and $T_{CQ}^{(1)}$ and compares their asymptotic powers, when the data dimension is large. Denote $\beta_{T_{S}}(\mu)$, $\beta_{T_{SR}}(\mu)$ and $\beta_{T_{CQ}^{(1)}}(\mu)$ to be the powers of the tests based on $T_{S}$, $T_{SR}$ and $T_{CQ}^{(1)}$ at a given level of significance, when the alternative hypothesis $H_{A} : \mu = {\bf 0}$ is true. Let us assume the following condition, which is the one sample version of condition (C3).
\vspace{0.02in}\\
{\it (C4) $||\mu||^{2}/d^{1/2+\epsilon} \rightarrow 0$ for some $\epsilon > 0$ and $\mu'\Sigma\mu = o({\mbox{tr}}(\Sigma^{2}))$ as $d \rightarrow \infty$, where $\Sigma = Disp({\bf X})$. }
\begin{theorem} \label{thm-new2}
Let ${\bf X} = \mu + {\bf V}$, where ${\bf V}$ is the vector formed by the first $d$ coordinates of the infinite sequence ${\cal V}$ satisfying conditions (C1) and (C2), and $\mu$ satisfies condition (C4). Define $\Gamma_{2} = 2{\mbox{tr}}(\Sigma^{2})/(n)_{2}$, and $\sigma^{2} = Var(X_{1})$, where ${\bf X} = (X_{1},X_{2},\ldots,X_{d})$. \\
(a) Each of $(d\sigma^{2}T_{S} - ||\mu||^{2})/\Gamma_{2}^{1/2}$, $(d\sigma^{2}T_{SR} - 2||\mu||^{2})/(2\Gamma_{2}^{1/2})$ and $(T_{CQ}^{(1)} - ||\mu||^{2})/\Gamma_{2}^{1/2}$ converges weakly to a standard Gaussian variable as $d \rightarrow \infty$ for every fixed $m,n \geq 1$. \\
(b) Assume $\lim_{d \rightarrow \infty} ||\mu||^{2}/\Gamma_{2}^{1/2} = c$ for some $c \in [0,\infty]$. Then, $\lim_{d \rightarrow \infty} \beta_{T_{S}}(\mu) = \lim_{d \rightarrow \infty} \beta_{T_{SR}}(\mu) = \lim_{d \rightarrow \infty} \beta_{T_{CQ}^{(1)}}(\mu) = \beta$ for every fixed $m,n \geq 1$, where $\beta = \alpha$, $\beta = 1$ or $\beta \in (\alpha,1)$ according as $c = 0$, $c = \infty$, or $c \in (0,\infty)$, respectively.
\end{theorem}
\indent We get the limiting null distributions of $T_{S}$, $T_{SR}$ and $T_{CQ}^{(1)}$ when $\mu = {\bf 0}$ in the above theorem. When both the data dimension and the sample size grow to infinity, \cite{WPL15} proved that the test based on $T_{S}$ is asymptotically as powerful as the test based on $T_{CQ}^{(1)}$ for spherical Gaussian distributions, which is a distribution included in our $\rho$-mixing model. The equality of the asymptotic powers of the tests based on $T_{S}$ and $T_{CQ}^{(1)}$ stated in part (b) of our Theorem \ref{thm-new2} holds for any sample size and for many non-spherical distributions.
\begin{remark}
In both the one and the two sample problems, when our $\rho$-mixing model for the data holds, the equality of the limiting powers of the tests based on sample means and the tests based on spatial signs and ranks, when the data dimension is large. This is true for any sample size and irrespective of whether the coordinate variables have Gaussian or some other heavy tailed distributions.
\end{remark}

\section{Asymptotic behaviours of different tests under stronger dependence}
\label{3}
\indent We now consider another class of probability models for high dimensional data, where there is stronger dependence among the coordinate variables than what we have considered in the previous section.
\begin{definition} \label{def2}
Consider an infinite sequence ${\cal X}$ defined over a probability space $(\Omega,{\cal A},P)$. We say that ${\cal X}$ is a randomly scaled $\rho$-mixing sequence (RSRM sequence, say) if there exist a zero mean $\rho$-mixing sequence ${\cal R}$ and a positive non-degenerate random variable $U$ defined on $(\Omega,{\cal A},P)$ such that ${\cal X} = {\cal R}/U$.
\end{definition}
\indent The RSRM property is satisfied by many important probability models for high dimensional data. For instance,  the infinite sequence of random variables associated with the multivariate spherical $t$ distribution has this property. In fact, by Theorem 1.31 in \cite{Kall05}, it follows that any rotatable sequence ${\cal X}$, i.e., a sequence for which all finite dimensional marginals are spherically symmetric, can be viewed as a RSRM sequence. Here,  ${\cal R}$ can be taken as a sequence of i.i.d. standard Gaussian variables and $U$ as a non-negative random variable independent of ${\cal R}$. More generally, if every finite dimensional marginal of a sequence ${\cal X}$ is elliptically symmetric, then ${\cal X} = {\cal R}/U$  with probability one, where ${\cal R}$ is a sequence of zero mean Gaussian variables, and $U$ is a non-negative random variable independent of ${\cal R}$. In this case, ${\cal X}$ has the RSRM property if the Gaussian sequence ${\cal R}$ is a $\rho$-mixing sequence. Let us mention here that \cite{WPL15} primarily worked under the setup of elliptically symmetric models, and from the above discussion it follows that this class includes many distributions that have the RSRM property. \cite{CLX14} also considered different classes of non-Gaussian models, and many of them have the RSRM property. \\
\indent For deriving the asymptotic distributions of $T_{WMW}$ and $T_{CQ}^{(2)}$ under the RSRM model, we assume the following.
\vspace{0.02in}\\
{\it (C5) ${\bf X} = \mu_{1} + \widetilde{{\bf V}}$ and ${\bf Y} = \mu_{2} + \widetilde{{\bf W}}$ for some $\mu_{1}, \mu_{2} \in \mathbb{R}^{d}$, where $\widetilde{{\bf V}}$ and $\widetilde{{\bf W}}$ are vectors formed by the first $d$ coordinates of RSRM sequences $\widetilde{{\cal V}}$ and $\widetilde{{\cal W}}$. Let $\widetilde{{\bf V}} = {\bf V}/P$ and $\widetilde{{\bf W}} = {\bf W}/Q$, where ${\cal V}$ and ${\cal W}$ are independent $\rho$-mixing sequences satisfying (C1) and (C2), and $P$ and $Q$ are independent positive random variables. } \vspace{0.02in}\\
As earlier, let ${\bf X}_{1},{\bf X}_{2},\ldots,{\bf X}_{m}$ and ${\bf Y}_{1},{\bf Y}_{2},\ldots,{\bf Y}_{n}$ be i.i.d. copies of independent random vectors ${\bf X}$ and ${\bf Y}$ in $\mathbb{R}^{d}$. Then, we can write ${\bf X}_{i} = \mu_{1} + {\bf V}_{i}/P_{i}$, $1 \leq i \leq m$, and ${\bf Y}_{j} = \mu_{2} + {\bf W}_{j}/Q_{j}$, $1 \leq j \leq n$.
\begin{theorem} \label{thm-new10}
Assume that (C5) holds, and $\mu = \mu_{2} - \mu_{1}$ satisfies condition (C3) with $\Sigma_{1}$ and $\Sigma_{2}$ in that condition replaced by $Disp({\bf V})$ and $Disp({\bf W})$, respectively. \\
(a) There exist random variables $S_{1}$, $S_{2}$ and $S_{3}$ that are functions of the $P_{i}$'s and the $Q_{j}$'s such that each of $(dT_{WMW} - ||\mu||^{2}S_{1})/S_{2}^{1/2}$ and $(T_{CQ}^{(2)} - ||\mu||^{2})/S_{3}^{1/2}$ converges weakly to a standard Gaussian variable as $d \rightarrow \infty$ for every $m,n \geq 1$. Consequently, for every fixed $m,n \geq 1$, the distributions of $T_{WMW}$ and $T_{CQ}^{(2)}$ can be approximated by location and scale mixtures of Gaussian distributions, when the data dimension is large. \\
(b) Assume further that all of $E(P), E(Q), E(P^{-2})$ and $E(Q^{-2})$ are finite, and $||\mu||^{2}/d^{1/2}$ tends to a finite non-negative limit as $d \rightarrow \infty$. Then, there exist real numbers $\psi_{1}$ and $\psi_{2}$ such that $\lim_{m,n \rightarrow \infty}\lim_{d \rightarrow \infty} P\{(dT_{WMW} - ||\mu||^{2}\psi_{1})/\psi_{2}^{1/2} \leq x\} = \lim_{m,n \rightarrow \infty}\lim_{d \rightarrow \infty} P\{(T_{CQ}^{(2)} - ||\mu||^{2})/\Gamma_{1}^{1/2} \leq x\} = \Phi(x)$ for all $x \in \mathbb{R}$. Here, $\Phi$ is the cumulative distribution function of standard Gaussian distribution, and $\Gamma_{1}$ is as defined in Theorem \ref{thm-new000}.
\end{theorem}
\indent Unlike the setup considered in Section \ref{2}, where the coordinate variables are $\rho$-mixing, here the distributions of $T_{WMW}$ and $T_{CQ}^{(2)}$ cannot be approximated by Gaussian distributions when $m$ and $n$ are small even if $d$ is large. However, if the sample sizes are also large in addition to data dimension, we can approximate the distributions of these statistics by Gaussian distributions. It is easy to see that many probability models with the RSRM property do not satisfy the model assumptions in (3.1) in \cite{CQ10}. Nevertheless, the asymptotic distribution of $T_{CQ}^{(2)}$ obtained from part (b) of Theorem \ref{thm-new10} coincides with that obtained in Theorem 1 in \cite{CQ10}. Further, it also coincides with the Gaussian distribution obtained under the $\rho$-mixing model in Theorem \ref{thm-new000}. \\
\indent Let $\beta_{T_{WMW}}(\mu)$ and $\beta_{T_{CQ}^{(2)}}(\mu)$ denote the powers of the tests based on $T_{WMW}$ and $T_{CQ}^{(2)}$ under the alternative hypothesis $H_{A} : \mu \neq {\bf 0}$ at a given level of significance. The next theorem gives a comparison of the asymptotic powers of these tests.
\begin{theorem}  \label{thm-new100}
Assume that ${\bf Y}$ has the same distribution as ${\bf X} + \mu$. Suppose that all the conditions assumed in Theorem \ref{thm-new10} hold. Also, assume that $\lim_{m,n \rightarrow \infty} \lim_{d \rightarrow \infty} ||\mu||^{2}/\Gamma_{1}^{1/2} = c$ for some $c \in (0,\infty)$. Then, $\lim_{m,n \rightarrow \infty} \lim_{d \rightarrow \infty} \beta_{T_{WMW}}(\mu) > \lim_{m,n \rightarrow \infty} \lim_{d \rightarrow \infty} \beta_{T_{CQ}^{(2)}}(\mu)$.
\end{theorem}
\indent If $\lim_{m,n \rightarrow \infty} \lim_{d \rightarrow \infty}||\mu||^{2}/\Gamma_{1}^{1/2}$ equals zero (respectively, infinity), then the asymptotic powers of the tests based on $T_{WMW}$ and $T_{CQ}^{(2)}$ in the setup of Theorem \ref{thm-new100} coincide, and they are both equal to the nominal level (respectively, equal to one). Theorem \ref{thm-new100} shows that for appropriate sequences of alternatives, the test based on $T_{WMW}$ is more powerful than the test based on $T_{CQ}^{(2)}$ for a large class of distributions including many spherical non-Gaussian distributions, when the data dimension as well as the sample sizes are large. Note that if ${\bf X}$ and ${\bf Y}$ have spherically symmetric distributions, then the conditions on $\mu$ in Theorems \ref{thm-new10} and \ref{thm-new100} hold if $\lim_{m,n \rightarrow \infty} \lim_{d \rightarrow \infty} (m+n)||\mu||^{2}/d^{1/2} = c' \in (0,\infty)$, and $\lim_{m,n \rightarrow \infty} m/(m+n) = \gamma \in (0,1)$.

\subsection{Empirical study using some RSRM models}
\label{3.1}
\indent The limiting null distribution of $T_{WMW}$ obtainable from Theorem \ref{thm-new10} cannot be used to implement this test because the parameters appearing in its limiting distribution cannot be estimated from the data. To compare the performances of the tests based on $T_{WMW}$ and $T_{CQ}^{(2)}$ for data from the spherical $t(5)$ distribution, we implemented these tests using their permutation distributions. Such an implementation has also been used by \cite{WLWM15} for their test. Though it is not possible to implement the test based on $T_{WMW}$ using its true asymptotic distribution in practice, we can do it for a simulation study, where the distributions and the associated parameters are known. On the other hand, since the true asymptotic null distribution of $T_{CQ}^{(2)}$ for RSRM models coincides with its asymptotic null distribution in the $\rho$-mixing setup, the implementation of this test can be done in the same way as described in subsection \ref{2.1}. We have chosen $m=n=20$, and $\mu = (c,0,0,\ldots,0)$ with $c = 1,1.5,2,2.5,3$ for $d = 100,200,400,800,1600$, respectively. Figure \ref{Fig2} shows that the sizes and the powers of these tests obtained by using the permutation implementation are not significantly different from the sizes and the powers of the tests implemented using their true asymptotic distributions. The permutation distributions of $T_{WMW}$ and $T_{CQ}^{(2)}$ adequately approximates their true distributions. Also, the test based on $T_{WMW}$ significantly outperforms the test based on $T_{CQ}^{(2)}$, which conforms with the result in Theorem \ref{thm-new100}.
\begin{figure}[ht!]
\begin{center}
\includegraphics[width=4.5in,height=1.5in]{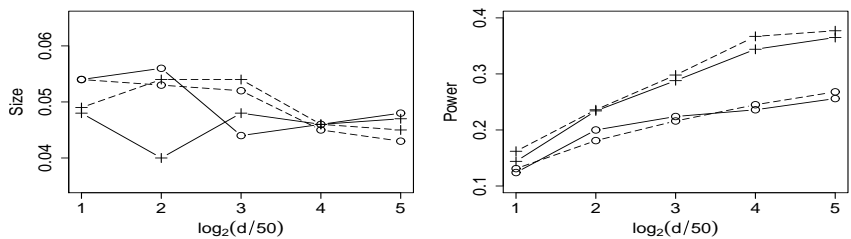}
\end{center}
\vspace{-0.15in}
\caption{Empirical sizes and powers of the tests based on $T_{WMW}$ ($+$) and $T_{CQ}^{(2)}$ ($\circ$) at nominal $5\%$ level for the spherical $t(5)$ distribution using the permutation implementation (solid curves) and the true implementation (dashed curves).  \label{Fig2}}
\end{figure}

\subsection{Asymptotic behaviours of one sample tests under stronger dependence}
\label{3.2}
\indent We will now study the asymptotic distributions of the one sample tests considered in subsection \ref{2.1} under the RSRM model. Let ${\bf X}_{1}, {\bf X}_{2}, \ldots, {\bf X}_{n}$ be i.i.d. copies of a random vector ${\bf X} \in \mathbb{R}^{d}$. The following theorem summarizes the asymptotic distributions of $T_{S}$, $T_{SR}$ and $T_{CQ}^{(1)}$ and yields their asymptotic powers. As earlier, we can write ${\bf X}_{i} = \mu + {\bf V}_{i}/P_{i}$, $1 \leq i \leq n$. Also, $\beta_{T_{S}}(\mu)$, $\beta_{T_{SR}}(\mu)$ and $\beta_{T_{CQ}^{(1)}}(\mu)$ denote the powers of the tests based on $T_{S}$, $T_{SR}$ and $T_{CQ}^{(1)}$ at a given level of significance, when the alternative hypothesis $H_{A} : \mu \neq {\bf 0}$ is true.
\begin{theorem} \label{thm-new3}
Let ${\bf X} = \mu + \widetilde{{\bf V}}$, where $\widetilde{{\bf V}}$ is the vector formed by the first $d$ coordinates of the sequence $\widetilde{{\cal V}}$ satisfying condition (C5), and $\mu$ satisfies condition (C4) with $\Sigma$ in that condition replaced by $Disp(\widetilde{{\bf V}})$. \\
(a) There exist $\Gamma_{3} > 0$ and random variables $Z_{k}$, $1 \leq k \leq 4$, which are functions of the $P_{i}$'s, such that each of $(dT_{S} - ||\mu||^{2}Z_{1})/\Gamma_{3}^{1/2}$, $(dT_{SR} - 2||\mu||^{2}Z_{2})/(2Z_{3}^{1/2})$ and $(T_{CQ}^{(1)} - ||\mu||^{2})/Z_{4}^{1/2}$ converges weakly to a standard Gaussian variable as $d \rightarrow \infty$ for each $m,n \geq 1$. Consequently, for each fixed $m,n \geq 1$, the distributions of $T_{S}$, $T_{SR}$ and $T_{CQ}^{(1)}$ are given by location and scale mixtures of Gaussian distributions, when the data dimension is large. \\
(b) Also, assume that both $E(P)$ and $E(P^{-2})$ are finite, and $||\mu||^{2}/d^{1/2}$ tends to a finite non-negative limit as $d \rightarrow \infty$. Define $\sigma^{2} = Var(X_{1})$. There exist real numbers $\theta_{k}$, $1 \leq k \leq 3$ such that $\lim_{n \rightarrow \infty}\lim_{d \rightarrow \infty} P\{(d\sigma^{2}T_{S} - ||\mu||^{2}\theta_{1})/\Gamma_{2}^{1/2} \leq x\} = \lim_{n \rightarrow \infty}\lim_{d \rightarrow \infty} P\{(d\sigma^{2}T_{SR} - 2||\mu||^{2}\theta_{2})/(2\theta_{3}^{1/2}) \leq x) = \lim_{n \rightarrow \infty}\lim_{d \rightarrow \infty} P\{(T_{CQ}^{(1)} - ||\mu||^{2})/\Gamma_{2}^{1/2} \leq x) = \Phi(x)$ for all $x \in \mathbb{R}$. Here, $\Phi$ denotes the cumulative distribution function of a standard Gaussian distribution, and $\Gamma_{2}$ is as defined in Theorem \ref{thm-new2}. \\
(c) Further, if we let $\lim_{n \rightarrow \infty} \lim_{d \rightarrow \infty} ||\mu||^{2}/\Gamma_{2}^{1/2} = c$, where $c \in (0,\infty)$, we have $\lim_{n \rightarrow \infty} \lim_{d \rightarrow \infty} \beta_{T_{S}}(\mu) > \lim_{n \rightarrow \infty} \lim_{d \rightarrow \infty} \beta_{T_{CQ}^{(1)}}(\mu)$. We also have $\lim_{n \rightarrow \infty} \lim_{d \rightarrow \infty} \beta_{T_{SR}}(\mu) > \lim_{n \rightarrow \infty} \lim_{d \rightarrow \infty} \beta_{T_{CQ}^{(1)}}(\mu)$.
\end{theorem}
\indent It is seen from the proof of part (a) of Theorem \ref{thm-new3} that if $E(P^{-2}) < \infty$, we have $\Gamma_{3} = \sigma^{-4}\Gamma_{2}$. In this case, we get the same limiting null distributions of $T_{S}$ from parts (a) and (b), i.e., its limiting null distribution is Gaussian irrespective of whether the sample size grows to infinity or not. Further, this limiting null distribution under the RSRM model is the same as that obtained under the $\rho$-mixing model in part (a) of Theorem \ref{thm-new10}. This is because the spatial sign $S({\bf x}) = {\bf x}/||{\bf x}||$, and thus $T_{S}$, remain invariant under homogeneous positive scale transformations of the coordinate variables. \\
\indent Note that the asymptotic distribution of $T_{CQ}^{(1)}$ is the same as that obtained in Theorem \ref{thm-new100} under the $\rho$-mixing setup, and it coincides with the asymptotic distribution of $T_{CQ}^{(1)}$ obtained by \cite{CQ10}. For the spherical $t$ distribution, which is a distribution included in our RSRM models, \cite{WPL15} derived the asymptotic distribution of $T_{CQ}^{(1)}$ and proved that the test based on $T_{S}$ is asymptotically more powerful than the former test. In the setup of Theorem \ref{thm-new3}, if $\lim_{n \rightarrow \infty} \lim_{d \rightarrow \infty} ||\mu||^{2}/\Gamma_{2}^{1/2}$ equals zero (respectively, infinity), then the asymptotic powers of the tests based on $T_{S}$, $T_{SR}$ and $T_{CQ}^{(1)}$ coincide, and they are all equal to the nominal level (respectively, equal to one).
\begin{remark} \label{rem2}
Suppose that in a two sample problem, ${\bf Y}$ is distributed as ${\bf X} + \mu$, where ${\bf X}$ is the vector formed by the first $d$ coordinates of a zero mean spherically symmetric or rotatable infinite sequence ${\cal X}$. Then, it follows from Theorem 1.31 in \cite{Kall05} that ${\bf X} = {\bf V}/P$, where ${\bf V}$ is a standard spherical Gaussian vector, and $P$ is a non-negative random variable independent of ${\bf V}$. Suppose that $\lim_{m,n \rightarrow \infty}\lim_{d \rightarrow \infty} (m+n)||\mu||^{2}/d^{1/2} = c' \in (0,\infty)$ and $\lim_{m,n \rightarrow \infty} m/(m+n) = \gamma \in (0,1)$. Also, assume that both $E(P)$ and $E(P^{-2})$ are finite and positive. Then, it follows from Theorems \ref{thm-new0} and \ref{thm-new100} that the test based on $T_{WMW}$ is asymptotically at least as powerful as the test based on $T_{CQ}^{(2)}$ if we first let the dimension and then the sample sizes grow to infinity. Further, their asymptotic powers are equal if and only if ${\bf X}$ has a spherical Gaussian distribution. In fact, in this case, their asymptotic powers are the same for any sample sizes if only the dimension grows to infinity.
\end{remark}
\begin{remark}  \label{rem3}
Suppose that in a one sample problem, we have ${\bf X} = \mu + \widetilde{{\bf V}}$, where $\widetilde{{\bf V}}$ is the vector formed by the first $d$ coordinates of a spherically symmetric infinite sequence. Assume that $\lim_{n \rightarrow \infty}\lim_{d \rightarrow \infty} n||\mu||^{2}/d^{1/2} = c' \in (0,\infty)$. Also, let both $E(P)$ and $E(P^{-2})$ be finite. Then, it follows from Theorems \ref{thm-new2} and \ref{thm-new3} that the tests based on $T_{S}$ and $T_{SR}$ are asymptotically at least as powerful as the test based on $T_{CQ}^{(1)}$ if we first let the dimension and then the sample size grow to infinity. Further, the asymptotic powers of all three tests are equal if and only if the distribution of ${\bf X}$ is spherical Gaussian. In fact, in this case, their asymptotic powers are the same for any sample size if only the dimension grows to infinity.
\end{remark}

\section{Analysis of real data}
\label{4}
\indent We now investigate the performances of the two sample tests based on $T_{WMW}$ and $T_{CQ}^{(2)}$ on some real datasets, when they are implemented in two different ways, namely, as in the $\rho$-mixing setup described in subsection \ref{2.2}, and using their permutation distributions. Two datasets are obtained from \path{http://www.cs.ucr.edu/~eamonn/time_series_data}, and the first of them is the ECG Data, which contains 69 normal ECG curves and 31 ECG curves of patients with a particular heart disease, and each curve is measured at 96 time points. The second data is the Gun Data, which contains the readings along the horizontal axis of the centroid of the right hand during two action sequences, namely, gun-draw and gun-point with 24 samples and 26 samples, respectively. Each action sequence is recorded at 150 time points. The third data is the Colon Data, which is obtained from \path{http://datam.i2r.a-star.edu.sg/datasets/krbd/ColonTumor/ColonTumor.zip} and contains the expression levels of $2000$ genes from $40$ tumor tissue and $22$ normal tissue. The fourth data is the Sonar Data obtained from \path{http://archive.ics.uci.edu/ml/datasets.html}, which contains sonar signals emitted from 111 metal cylinder samples and 97 rock samples, and each signal is recorded at 60 wavelengths. To estimate the sizes of the tests based on $T_{WMW}$ and $T_{CQ}^{(2)}$ for each data, we selected two random subsamples $1000$ times from one class in that data and computed the proportion of rejections for each test. The same procedure is now repeated for the other class and the two values obtained for each test are averaged. For evaluating the powers of these tests, we selected $1000$ random subsamples each from the two classes and computed the proportions of rejections for the tests. The size of each subsample is $20\%$, $40\%$, $40\%$ and $20\%$ of the original sample size for the ECG Data, the Gun Data, the Colon Data and the Sonar Data, respectively. These choices are made to ensure that the resulting datasets remain
high dimensional, and the powers of the tests are neither too close to the nominal $5\%$ level nor to one. For computing the permutation distributions of the test statistics, we have used $500$ random permutations of the two subsamples.

\begin{table}[ht!]
\centering
\caption{Sizes and powers of the tests based on $T_{WMW}$ and $T_{CQ}^{(2)}$ at nominal $5\%$ level for some real data.  \label{Tab1}}
\begin{tabular}{rrrrrrrrr}
\hline
Data $\longrightarrow$ & \multicolumn{2}{c}{ECG} & \multicolumn{2}{c}{Gun} & \multicolumn{2}{c}{Colon} & \multicolumn{2}{c}{Sonar} \\ \hline
\multicolumn{9}{c}{Implementation as in the $\rho$-mixing setup} \\ \hline
& Size & Power & Size & Power & Size & Power & Size & Power \\
$T_{WMW}$ & 0.052 & 0.593 & 0.052 & 0.501 & 0.056 & 0.747 & 0.036 & 0.507 \\
$T_{CQ}^{(2)}$ & 0.063 & 0.601 & 0.058 & 0.500 & 0.063 & 0.641 & 0.058 & 0.432 \\ \hline
\multicolumn{9}{c}{Permutation implementation} \\ \hline
& Size & Power & Size & Power & Size & Power & Size & Power \\
$T_{WMW}$ & 0.057 & 0.643 & 0.055 & 0.472 & 0.055 & 0.723 & 0.043 & 0.519 \\
$T_{CQ}^{(2)}$ & 0.057 & 0.624 & 0.052 & 0.442 & 0.060 & 0.596 & 0.038 & 0.360 \\ \hline
\end{tabular}
\end{table}

\indent Table \ref{Tab1} shows that the sizes as well as the powers of the tests for the two implementations are not significantly different. However, the permutation implementation required almost ten times more computing time. Moreover, the sizes of the tests are close to the nominal $5\%$ level for all the four datasets. Further, the powers of the tests based on $T_{WMW}$ and $T_{CQ}^{(2)}$ are not significantly different for the ECG data and the Gun data. However, the test based on $T_{WMW}$ is significantly more powerful than the test based on $T_{CQ}^{(2)}$ for the Colon data and the Sonar data.

\section{Concluding remarks and discussion}
\label{5}
\indent  We now consider the performances of some other mean based tests studied in the literature and discussed in Section \ref{2} on some simulated datasets. We denote the test statistics associated with the tests in \cite{SKK13} and \cite{GCBL14} by $T_{SKK}$ and $T_{GCBL}$, respectively. For the $AR(1)$ models in subsection \ref{2.1}, we found that the size of the test based on $T_{SKK}$ increases with $d$ and becomes significantly larger than the nominal $5\%$ level for $d \geq 400$. \cite{FZWZ15} proved that the size of this test converges to one as the dimension and the sample sizes grow to infinity at a certain rate for a class of models, which include these $AR(1)$ models. Under the spherical $t(5)$ model in subsection \ref{3.1}, the size of the test based on $T_{SKK}$ is significantly less than the nominal level for all values of $d$ considered and decreases to zero as $d$ increases. The size of the test based on $T_{GCBL}$ is significantly larger than the nominal level for all values of $d$ considered under the $AR(1)$ models as well as the spherical $t(5)$ model. It seems that the estimates of the critical values for the tests based on $T_{SKK}$ and $T_{GCBL}$ are adversely affected if the sample size is much smaller than the dimension as in our simulation study. On the other hand, we found that permutation implementations of these tests correct their sizes under all of the above models. Even then, these tests are significantly less powerful than the test based on $T_{WMW}$ (respectively, $T_{CQ}^{(2)}$) under all the above models (respectively, $AR(1)$ models) but they outperform the test based on $T_{CQ}^{(2)}$ under the spherical $t(5)$ model. The readers are referred to Appendix -- III for more details.  \\
\indent \cite{CLX14} showed that their test has better power than other tests based on sum of squares of coordinatewise mean difference or coordinatewise $t$ statistics, when the mean shift has only a few non-zero coordinates. However, we observed that this test becomes significantly less powerful than the tests based on $T_{WMW}$ and $T_{CQ}^{(2)}$, when the mean shifts in the models considered in subsections \ref{2.1} and \ref{3.1} are distributed equally among all the coordinates. Moreover, the size of the test in \cite{CLX14} increases with $d$ and becomes significantly larger than the nominal level for $d \geq 400$ under all of the above models. It seems that the asymptotic extreme value distribution of this statistic is not adequate if the data dimension is much larger than the sample size. Since the test in \cite{CLX14} involves a computationally intensive optimization involving sample dispersion matrices, we could not implement this test using the permutation approach. The detailed results of the simulation study are provided in Appendix -- III. \\
\indent Multivariate Gaussian distributions with dispersion matrices of the form $(1-\beta)I_{d} + \beta{\bf 1}_{d}{\bf 1}_{d}'$ for some $\beta \in (0,1)$, where ${\bf 1}_{d}$ denotes the $d$-dimensional vector of one's, are neither $\rho$-mixing nor have the RSRM property. Recently, \cite{KK14} mentioned that for such probability models for high dimensional data, the size of test based on $T_{CQ}^{(2)}$ would be asymptotically incorrect. To compare the performance of the tests based on $T_{WMW}$ and $T_{CQ}^{(2)}$ for such models, we have chosen $\beta = 0.7$, $m=n=20$ and used the permutation implementations of these tests. The mean shifts chosen are $\mu = (c,0,0,\ldots,0)$ with $c = 2.5,5,7.5,10,12.5$ for $d = 100,200,400,800,1600$, respectively. We found that the test based on $T_{WMW}$ significantly outperforms the test based on $T_{CQ}^{(2)}$ for all values of $d$ (see Appendix -- III).

\section*{Appendix -- I}
\begin{proof}[Proof of Theorem \ref{thm-new000}]
Without any loss of generality, we can take $E({\bf X}_{1}) = {\bf 0}$. Let us write ${\bf X}_{i} = (X_{i1},X_{i2},\ldots,X_{id})'$, $1 \leq i \leq m$, and ${\bf Y}_{j} = (Y_{j1},Y_{j2},\ldots,Y_{jd})'$, $1 \leq j \leq n$. First note that
\begin{eqnarray}
||{\bf X} - {\bf Y}||^{2} &=& [||{\bf X}||^{2} + ||{\bf Y} - \mu||^{2} - 2{\bf X}'({\bf Y} - \mu) + 2\mu'({\bf X} - {\bf Y} + \mu) + ||\mu||^{2}] \nonumber \\
&=& \sum_{k=1}^{d} [V_{k}^{2} + W_{k}^{2} - 2V_{k}W_{k}] + 2\mu'({\bf V} - {\bf W}) + ||\mu||^{2},  \label{thm1a-eq1}
\end{eqnarray}
It follows from \citet[Theorem 5.2(b)]{Brad05} that for any function $h : \mathbb{R}^{2} \rightarrow \mathbb{R}$, the sequence $(h(V_{k},W_{k}) : k \geq 1)$ is $\rho$-mixing with its mixing coefficient bounded by $\max\{\rho_{1}(\cdot),\rho_{2}(\cdot)\}$. This fact, \eqref{thm1a-eq1} above along with Assumptions (C1)--(C3) and Theorem 8.2.2 in \cite{LL96} imply that for any given $\epsilon \in (0,1/2)$, we have
\begin{eqnarray}
||{\bf X} - {\bf Y}||^{2}/d - (\sigma_{1}^{2} + \sigma_{2}^{2}) = o(d^{-1/2 + \epsilon}) \label{eq1.0}
\end{eqnarray}
as $d \rightarrow \infty$ {\it almost surely}. Now,
\begin{eqnarray}
T_{WMW} &=& \frac{1}{(m)_{2}(n)_{2}} \sum_{i_{1} \neq i_{2}}\sum_{j_{1} \neq j_{2}} \frac{({\bf
X}_{i_{1}} - {\bf Y}_{j_{1}})'({\bf X}_{i_{2}} - {\bf Y}_{j_{2}})}{d(\sigma_{1}^{2}+\sigma_{2}^{2})} \nonumber \\
&& + \ \frac{1}{(m)_{2}(n)_{2}} \sum_{i_{1} \neq i_{2}}\sum_{j_{1} \neq j_{2}} \left[\frac{({\bf
X}_{i_{1}} - {\bf Y}_{j_{1}})'({\bf X}_{i_{2}} - {\bf Y}_{j_{2}})}{d(\sigma_{1}^{2}+\sigma_{2}^{2})} \times  \left\{\frac{d(\sigma_{1}^{2}+\sigma_{2}^{2})}{||{\bf X}_{i_{1}} - {\bf Y}_{j_{1}}||~||{\bf X}_{i_{2}} - {\bf Y}_{j_{2}}||} - 1\right\}\right] \nonumber \\
&=& (T_{CQ}^{(2)} + T_{WMW}^{(2)})/\{d(\sigma_{1}^{2}+\sigma_{2}^{2})\}, \label{eq1.1}
\end{eqnarray}
where $T_{CQ}^{(2)} = [(m)_{2}(n)_{2}]^{-1} \sum_{i_{1} \neq i_{2}}\sum_{j_{1} \neq j_{2}} ({\bf
X}_{i_{1}} - {\bf Y}_{j_{1}})'({\bf X}_{i_{2}} - {\bf Y}_{j_{2}})$ as defined in the Introduction, and
\begin{eqnarray*}
&& T_{WMW}^{(2)} = \frac{1}{(m)_{2}(n)_{2}} \sum_{i_{1} \neq i_{2}}\sum_{j_{1} \neq j_{2}} \left[\frac{({\bf
X}_{i_{1}} - {\bf Y}_{j_{1}})'({\bf X}_{i_{2}} - {\bf Y}_{j_{2}})}{d(\sigma_{1}^{2}+\sigma_{2}^{2})} \times \left\{\frac{d(\sigma_{1}^{2}+\sigma_{2}^{2})}{||{\bf X}_{i_{1}} - {\bf Y}_{j_{1}}||~||{\bf X}_{i_{2}} - {\bf Y}_{j_{2}}||} - 1\right\}\right].
\end{eqnarray*}
So, $E(T_{CQ}^{(2)}) = ||\mu||^{2}$. Further, it follows from \citet[p. 825]{CQ10} that $Var(T_{CQ}^{(2)}) = \Gamma_{1} + 4\mu'\Sigma_{1}\mu/m + 4\mu'\Sigma_{2}\mu/n$, where $\Gamma_{1} = 2{\mbox{tr}}(\Sigma_{1}^{2})/(m)_{2} + 2{\mbox{tr}}(\Sigma_{2}^{2})/(n)_{2} + 4{\mbox{tr}}(\Sigma_{1}\Sigma_{2})/(mn)$ is defined in the statement of the theorem. Note that $(\mu'\Sigma_{1}\mu/m) + (\mu'\Sigma_{2}\mu/n) \leq \mu'(\Sigma_{1} + \Sigma_{2})\mu/\min(m,n)$. Also, the denominator of each of the three terms in $\Gamma_{1}$ is less than $(N)_{2}$, where $N = \max(m,n)$. This implies that $\Gamma_{1} \geq [2{\mbox{tr}}(\Sigma_{1}^{2}) + 2{\mbox{tr}}(\Sigma_{2}^{2}) + 4{\mbox{tr}}(\Sigma_{1}\Sigma_{2})]/(N)_{2} = 2{\mbox{tr}}[(\Sigma_{1} + \Sigma_{2})^{2}]/(N)_{2}$. These facts and Assumption (C3) imply that $Var(T_{CQ}^{(2)}) = \Gamma_{1}(1+o(1))$ as $d \rightarrow \infty$. Further,
\begin{eqnarray*}
T_{CQ}^{(2)} - ||\mu||^{2} &=& \frac{1}{(m)_{2}(n)_{2}}\sum_{i_{1} \neq i_{2}} \sum_{j_{1} \neq j_{2}} ({\bf X}_{i_{1}} - {\bf Y}_{j_{1}} + \mu)'({\bf X}_{i_{2}} - {\bf Y}_{j_{2}} + \mu) - \frac{2}{mn} \sum_{i,j} \mu'({\bf X}_{i} - {\bf Y}_{j} + \mu) \\
&=& T_{1} - T_{2},
\end{eqnarray*}
where $T_{1} = [(m)_{2}(n)_{2}]^{-1}\sum_{i_{1} \neq i_{2}} \sum_{j_{1} \neq j_{2}} ({\bf X}_{i_{1}} - {\bf Y}_{j_{1}} + \mu)'({\bf X}_{i_{2}} - {\bf Y}_{j_{2}} + \mu)$ and $T_{2} = 2(mn)^{-1}\sum_{i,j} \mu'({\bf X}_{i} - {\bf Y}_{j} + \mu)$. It is easy to verify that $E(T_{2}) = 0$ and $Var(T_{2}) = 4\mu'[(\Sigma_{1}/m)+(\Sigma_{2}/n)]\mu$. So, using the inequality $\Gamma_{1} \geq 2{\mbox{tr}}[(\Sigma_{1} + \Sigma_{2})^{2}]/(N)_{2}$, Assumption (C3) and  Chebyshev's inequality, it follows that $T_{2}/\Gamma_{1}^{1/2}$ converges to zero {\it in probability} as $d \rightarrow \infty$. Note that
\begin{eqnarray*}
T_{1} = \frac{1}{(m)_{2}(n)_{2}} \sum_{k=1}^{d} \left\{ \sum_{i_{1} \neq i_{2}} \sum_{j_{1} \neq j_{2}} (V_{i_{1}k} - W_{j_{1}k})(V_{i_{2}k} - W_{j_{2}k}) \right\}. \\
\end{eqnarray*}
So, $E(T_{1}) = 0$ and $Var(T_{1}) = \Gamma_{1}$. This follows from computations similar to those used in deriving $Var(T_{CQ}^{(2)})$ earlier. Thus, by Theorem 4.0.1 in \cite{LL96} and Assumptions (C1) and (C2), we have the {\it weak convergence} of $T_{1}/\Gamma_{1}^{1/2}$ to a standard Gaussian distribution as $d \rightarrow \infty$ for each fixed $m,n \geq 1$. This and the fact that $T_{2}^{(2)}$ converges to zero {\it in probability} as $d \rightarrow \infty$ for each fixed $m,n \geq 1$ together imply that
\begin{eqnarray}
(T_{CQ}^{(2)} - ||\mu||^{2})/\Gamma_{1}^{1/2} \stackrel{{\cal L}}{\longrightarrow} N(0,1) \label{eq1.2}
\end{eqnarray}
as $d \rightarrow \infty$ for each fixed $m,n \geq 1$. Next, let us write
\begin{eqnarray}
T_{WMW}^{(2)}/\Gamma_{1}^{1/2} &=& \frac{1}{(m)_{2}(n)_{2}} \sum_{i_{1} \neq i_{2}} \sum_{j_{1} \neq j_{2}} \left[\frac{({\bf X}_{i_{1}} - {\bf Y}_{j_{1}})'({\bf X}_{i_{2}} - {\bf Y}_{j_{2}}) - ||\mu||^{2}}{\Gamma_{1}^{1/2}} \times \right. \nonumber \\
&& \hspace{2cm} \left.\left\{\frac{d(\sigma_{1}^{2}+\sigma_{2}^{2})}{||{\bf X}_{i_{1}} - {\bf Y}_{j_{1}}||~||{\bf X}_{i_{2}} - {\bf Y}_{j_{2}}||} - 1\right\}\right]  \nonumber \\
&& + \ \frac{||\mu||^{2}}{(m)_{2}(n)_{2}\Gamma_{1}^{1/2}} \sum_{i_{1} \neq i_{2}} \sum_{j_{1} \neq j_{2}} \left\{\frac{d(\sigma_{1}^{2}+\sigma_{2}^{2})}{||{\bf X}_{i_{1}} - {\bf Y}_{j_{1}}||~||{\bf X}_{i_{2}} - {\bf Y}_{j_{2}}||} - 1\right\} \nonumber \\
&=& T_{WMW}^{(3)} + T_{WMW}^{(4)}  \label{eq1.1.2},
\end{eqnarray}
where
\begin{eqnarray*}
&& T_{WMW}^{(3)} = \frac{1}{(m)_{2}(n)_{2}} \sum_{i_{1} \neq i_{2}} \sum_{j_{1} \neq j_{2}} \left[\frac{({\bf X}_{i_{1}} - {\bf Y}_{j_{1}})'({\bf X}_{i_{2}} - {\bf Y}_{j_{2}}) - ||\mu||^{2}}{\Gamma_{1}^{1/2}} \times \right. \\
&& \hspace{3cm}\left.\left\{\frac{d(\sigma_{1}^{2}+\sigma_{2}^{2})}{||{\bf X}_{i_{1}} - {\bf Y}_{j_{1}}||~||{\bf X}_{i_{2}} - {\bf Y}_{j_{2}}||} - 1\right\}\right] \ \ \mbox{and}\\
&& T_{WMW}^{(4)} = \frac{||\mu||^{2}}{(m)_{2}(n)_{2}\Gamma_{1}^{1/2}} \sum_{i_{1} \neq i_{2}} \sum_{j_{1} \neq j_{2}} \left\{\frac{d(\sigma_{1}^{2}+\sigma_{2}^{2})}{||{\bf X}_{i_{1}} - {\bf Y}_{j_{1}}||~||{\bf X}_{i_{2}} - {\bf Y}_{j_{2}}||} - 1\right\}.
\end{eqnarray*}
As mentioned earlier, $\Gamma_{1} \geq 2{\mbox{tr}}[(\Sigma_{1} + \Sigma_{2})^{2}]/(N)_{2}$. Also, from the stationarity of the sequences ${\cal X}$ and ${\cal Y}$ and using the Cauchy-Schwarz inequality, it follows that ${\mbox{tr}}[(\Sigma_{1}+\Sigma_{2})^{2}] \geq d(\sigma_{1}^{2}+\sigma_{2}^{2})^{2}$. These facts along with \eqref{eq1.0} and Assumption (C3) imply that each term inside the double summation appearing in the definition of $T_{WMW}^{(4)}$ converges to zero {\it in probability} as $d \rightarrow \infty$ for each fixed $m,n \geq 1$. So, $T_{WMW}^{(4)}$ converges to zero {\it in probability} as $d \rightarrow \infty$ for each fixed $m,n \geq 1$. \\
\indent Next, fix any $i_{1} \neq i_{2}$ and $j_{1} \neq j_{2}$ and consider the corresponding term inside the double summation appearing in the definition of $T_{WMW}^{(3)}$. It follows from \eqref{eq1.0} that $d(\sigma_{1}^{2}+\sigma_{2}^{2})/[||{\bf X}_{i_{1}} - {\bf Y}_{j_{1}}||~||{\bf X}_{i_{2}} - {\bf Y}_{j_{2}}||] - 1$ converges to zero {\it in probability} as $d \rightarrow \infty$. Also, note that
\begin{eqnarray}
&& ({\bf X}_{i_{1}} - {\bf Y}_{j_{1}})'({\bf X}_{i_{2}} - {\bf Y}_{j_{2}}) - ||\mu||^{2} = ({\bf X}_{i_{1}} - {\bf Y}_{j_{1}} + \mu)'({\bf X}_{i_{2}} - {\bf Y}_{j_{2}} + \mu) \label{thm1a-eq2} \\
&& \hspace{2cm}- \ \mu'({\bf X}_{i_{1}} - {\bf Y}_{j_{1}} + \mu) - \mu'({\bf X}_{i_{2}} - {\bf Y}_{j_{2}} + \mu). \nonumber
\end{eqnarray}
Using arguments similar to those used to prove the asymptotic normality of $T_{1}^{(2)}$ and using Theorem 4.0.1 in \cite{LL96}, it follows that the first term in the right hand side of \eqref{thm1a-eq2} is asymptotically Gaussian with zero mean and variance $2{\mbox{tr}}[(\Sigma_{1}+\Sigma_{2})^{2}]$ as $d \rightarrow \infty$. Using Assumption (C3) and Chebyshev's inequality, it follows that the second and the third terms in the right hand side of \eqref{thm1a-eq2} after scaling by $\Gamma_{1}^{1/2}$ converge to zero {\it in probability} as $d \rightarrow \infty$. So, the left hand side of \eqref{thm1a-eq2} after scaling by $\Gamma_{1}^{1/2}$ converges {\it weakly} to a Gaussian distribution as $d \rightarrow \infty$. Thus, $T_{WMW}^{(3)}$ converges to zero {\it in probability} as $d \rightarrow \infty$ for each fixed $m,n \geq 1$. This and the fact that $T_{WMW}^{(4)}$ converges to zero {\it in probability} as $d \rightarrow \infty$ together imply that $T_{WMW}^{(2)}/\Gamma_{1}^{1/2}$ converges to zero {\it in probability} as $d \rightarrow \infty$ for each fixed $m,n \geq 1$. Combining this fact with \eqref{eq1.1} and \eqref{eq1.2} yields
\begin{eqnarray}
\{d(\sigma_{1}^{2}+\sigma_{2}^{2})T_{WMW} - ||\mu||^{2}\}/\Gamma_{1}^{1/2} \stackrel{{\cal L}}{\rightarrow} N(0,1)  \label{eq1.2.1}
\end{eqnarray}
as $d \rightarrow \infty$ for each fixed $m,n \geq 1$.
\end{proof}

\begin{proof}[Proof of Theorem \ref{thm-new0}]
Let $\zeta_{\alpha}$ be the $(1-\alpha)$-quantile of the standard Gaussian distribution. Note that
\begin{eqnarray*}
\beta_{T_{WMW}}(\mu) &=& P\{d(\sigma_{1}^{2}+\sigma_{2}^{2})T_{WMW}/\Gamma_{1}^{1/2} > \zeta_{\alpha}\} \\
&=& P\{[d(\sigma_{1}^{2}+\sigma_{2}^{2})T_{WMW} - ||\mu||^{2}]/\Gamma_{1}^{1/2} > \zeta_{\alpha} - ||\mu||^{2}/\Gamma_{1}^{1/2}\}
\end{eqnarray*}
and
$$\beta_{T_{CQ}^{(2)}}(\mu) = P\{T_{CQ}^{(2)}/\Gamma_{1}^{1/2} > \zeta_{\alpha}\} = P\{(T_{CQ}^{(2)} - ||\mu||^{2})/\Gamma_{1}^{1/2} > \zeta_{\alpha} - ||\mu||^{2}/\Gamma_{1}^{1/2}\},$$
where the probabilities are computed under the alternative hypothesis. Since $\lim_{d \rightarrow \infty} ||\mu||^{2}/\Gamma_{1}^{1/2}$ exists, the equality of the asymptotic powers of the tests based on $T_{WMW}$ and $T_{CQ}^{(2)}$ follows from \eqref{eq1.2} and \eqref{eq1.2.1}. Moreover, their common value is $\Phi(-\zeta_{\alpha} + \lim_{d \rightarrow \infty}||\mu||^{2}/\Gamma_{1}^{1/2}) = \Phi(-\zeta_{\alpha} + c)$, which follows from the expressions of their powers and their asymptotic Gaussian distributions proved in Theorem \ref{thm-new000}. The last part of the present theorem now follows easily.
\end{proof}

\begin{proof}[Proof of Theorem \ref{thm-new2}]
(a) We will derive the asymptotic distribution of $T_{SR}$ and $T_{CQ}^{(1)}$ only, since the derivation of that of $T_{S}$ is simpler and follows from similar arguments. Using the assumptions in the theorem and the arguments similar to those in the proof of Theorem \ref{thm-new000}, we have
\begin{eqnarray}
T_{SR} &=& \frac{1}{(n)_{4}} \sum_{i_{1} \neq i_{2} \neq i_{3} \neq i_{4}} \frac{({\bf X}_{i_{1}} + {\bf X}_{i_{2}})'({\bf X}_{i_{3}} + {\bf X}_{i_{4}})}{2d\sigma^{2}} \nonumber
\\
&& + \ \frac{1}{(n)_{4}} \sum_{i_{1} \neq i_{2} \neq i_{3} \neq i_{4}} \left[\frac{({\bf X}_{i_{1}} + {\bf X}_{i_{2}})'({\bf X}_{i_{3}} + {\bf X}_{i_{4}})}{2d\sigma^{2}} \times \left\{\frac{2d\sigma^{2}}{||{\bf X}_{i_{1}} + {\bf X}_{i_{2}}||~||{\bf X}_{i_{3}} + {\bf X}_{4}||} - 1\right\}\right] \nonumber \\
&=& \frac{2}{(n)_{2}} \sum_{i_{1} \neq i_{2}} \frac{{\bf X}_{i_{1}}'{\bf X}_{i_{2}}}{d\sigma^{2}} + \frac{1}{(n)_{4}} \sum_{i_{1} \neq i_{2} \neq i_{3} \neq i_{4}} \left[\frac{({\bf X}_{i_{1}} + {\bf X}_{i_{2}})'({\bf X}_{i_{3}} + {\bf X}_{i_{4}})}{2d\sigma^{2}} \times \right. \label{eq5.1} \\
&& \hspace{2cm} \left.\left\{\frac{2d\sigma^{2}}{||{\bf X}_{i_{1}} + {\bf X}_{i_{2}}||~||{\bf X}_{i_{3}} + {\bf X}_{4}||} - 1\right\}\right] \nonumber
\end{eqnarray}
The first term in \eqref{eq5.1} equals $2T_{CQ}^{(1)}/(d\sigma^{2})$. Using Assumption (C4), it can be shown that $E(T_{CQ}^{(1)}) = ||\mu||^{2}$ and $Var(T_{CQ}^{(1)}) = \Gamma_{2}(1 + o(1))$ as $d \rightarrow \infty$. Using arguments similar to those in the proof of Theorem \ref{thm-new000}, we have the {\it weak convergence} of $(T_{CQ}^{(1)} - ||\mu||^{2})/\psi_{2}^{1/2}$ to a standard Gaussian distribution. Further, the second term in \eqref{eq5.1} after scaling by $\Gamma_{2}^{1/2}$ converges to zero {\it in probability} as $d \rightarrow \infty$ for each $n \geq 1$. The previous two statements together imply that $(d\sigma^{2}T_{SR} - 2||\mu||^{2})/\psi_{2}^{1/2}$ converges {\it weakly} to a $N(0,4)$ distribution as $d \rightarrow \infty$ for each $n \geq 1$. \\
(b) The proof of this part of the theorem follows from arguments similar to those used in the proof of Theorem \ref{thm-new0}.
\end{proof}

\begin{proof}[Proof of Theorem \ref{thm-new10}]
Without any loss of generality, we can take $\mu_{1} = {\bf 0}$, so that $\mu = \mu_{2}$. Let us write ${\bf X}_{i} = \widetilde{{\bf V}}_{i}$ and ${\bf Y}_{j} = \mu + \widetilde{{\bf W}}_{j}$, where $\widetilde{{\bf V}}_{i} = {\bf V}_{i}/P_{i}$ and $\widetilde{{\bf W}}_{j} = {\bf W}_{j}/Q_{j}$ for $1 \leq i \leq m$ and $i \leq j \leq n$. Let ${\bf V} = (V_{1},V_{2},\ldots,V_{d})'$ and ${\bf W} = (W_{1},W_{2},\ldots,W_{d})'$. Denote $\Sigma_{V} = Disp({\bf V})$, $\Sigma_{W} = Disp({\bf W})$, $\sigma_{V}^{2} = Var(V_{1})$ and $\sigma_{W}^{2} = Var(V_{2})$. \\
(a) We will first derive the asymptotic distribution of $T_{WMW}$. Using similar arguments as those used in proving \eqref{thm1a-eq1}, we get
\begin{eqnarray}
&& ||{\bf X} - {\bf Y}||^{2} = \sum_{k=1}^{d} \left[\frac{V_{k}^{2}}{P^{2}} + \frac{W_{k}^{2}}{Q^{2}} - \frac{2V_{k}W_{k}}{PQ}\right] + 2\mu'\left(\frac{{\bf V}}{P} - \frac{{\bf W}}{Q}\right) + ||\mu||^{2}.  \label{eq1}
\end{eqnarray}
Consider the event $E = \{||{\bf X} - {\bf Y}||^{2}/d - (\sigma_{V}^{2}/P^{2} + \sigma_{W}^{2}/Q^{2}) = o(d^{-1/2 + \epsilon}) ~\mbox{as}~d \rightarrow \infty\}$. It follows from \citet[Theorem 5.2(b)]{Brad05} that for any function $h : \mathbb{R}^{2} \rightarrow \mathbb{R}$, the sequence $(h(V_{k},W_{k}) : k \geq 1)$ is $\rho$-mixing with its mixing coefficient bounded by $\max\{\rho_{1}(\cdot),\rho_{2}(\cdot)\}$. Using this fact and \eqref{eq1} above along with the assumptions in the theorem and Theorem 8.2.2 in \cite{LL96}, we get that for any given $\epsilon \in (0,1/2)$,
\begin{eqnarray}
Pr(E | P, Q) = 1 \label{eq1-1}
\end{eqnarray}
for {\it almost every} $P$ and $Q$. Now,
\begin{eqnarray}
T_{WMW} &=& \frac{1}{(m)_{2}(n)_{2}} \sum_{i_{1} \neq i_{2}} \sum_{j_{1} \neq j_{2}}
\frac{({\bf X}_{i_{1}} - {\bf Y}_{j_{1}})'({\bf X}_{i_{2}} - {\bf Y}_{j_{2}})}{d(\sigma_{V}^{2}P_{i_{1}}^{-2} + \sigma_{W}^{2}Q_{j_{1}}^{-2})^{1/2}(\sigma_{V}^{2}P_{i_{2}}^{-2} + \sigma_{W}^{2}Q_{j_{2}}^{-2})^{1/2}}  \nonumber \\
&& + \ \frac{1}{(m)_{2}(n)_{2}} \sum_{i_{1} \neq i_{2}} \sum_{j_{1} \neq j_{2}} \left[
\frac{({\bf X}_{i_{1}} - {\bf Y}_{j_{1}})'({\bf X}_{i_{2}} - {\bf Y}_{j_{2}})}{d(\sigma_{V}^{2}P_{i_{1}}^{-2} + \sigma_{W}^{2}Q_{j_{1}}^{-2})^{1/2}(\sigma_{V}^{2}P_{i_{2}}^{-2} + \sigma_{W}^{2}Q_{j_{2}}^{-2})^{1/2}} \times \right.  \nonumber  \\
&& \left. \left\{\frac{d(\sigma_{V}^{2}P_{i_{1}}^{-2} + \sigma_{W}^{2}Q_{j_{1}}^{-2})^{1/2}(\sigma_{V}^{2}P_{i_{2}}^{-2} + \sigma_{W}^{2}Q_{j_{2}}^{-2})^{1/2}}{||{\bf X}_{i_{1}} - {\bf Y}_{j_{1}}||~||{\bf X}_{i_{2}} - {\bf Y}_{j_{2}}||} - 1\right\}\right].   \nonumber \\
&=& (T_{WMW}^{(1)} + T_{WMW}^{(2)})/d,  \label{eq2}
\end{eqnarray}
where
\begin{eqnarray*}
T_{WMW}^{(1)} = \frac{1}{(m)_{2}(n)_{2}} \sum_{i_{1} \neq i_{2}} \sum_{j_{1} \neq j_{2}}
\frac{({\bf X}_{i_{1}} - {\bf Y}_{j_{1}})'({\bf X}_{i_{2}} - {\bf Y}_{j_{2}})}{d(\sigma_{V}^{2}P_{i_{1}}^{-2} + \sigma_{W}^{2}Q_{j_{1}}^{-2})^{1/2}(\sigma_{V}^{2}P_{i_{2}}^{-2} + \sigma_{W}^{2}Q_{j_{2}}^{-2})^{1/2}}
\end{eqnarray*}
and
\begin{eqnarray*}
T_{WMW}^{(2)} = \frac{1}{(m)_{2}(n)_{2}} \sum_{i_{1} \neq i_{2}} \sum_{j_{1} \neq j_{2}} \left[
\frac{({\bf X}_{i_{1}} - {\bf Y}_{j_{1}})'({\bf X}_{i_{2}} - {\bf Y}_{j_{2}})}{d(\sigma_{V}^{2}P_{i_{1}}^{-2} + \sigma_{W}^{2}Q_{j_{1}}^{-2})^{1/2}(\sigma_{V}^{2}P_{i_{2}}^{-2} + \sigma_{W}^{2}Q_{j_{2}}^{-2})^{1/2}} \times \right.\\
\hspace{2cm} \left.\left\{\frac{d(\sigma_{V}^{2}P_{i_{1}}^{-2} + \sigma_{W}^{2}Q_{j_{1}}^{-2})^{1/2}(\sigma_{V}^{2}P_{i_{2}}^{-2} + \sigma_{W}^{2}Q_{j_{2}}^{-2})^{1/2}}{||{\bf X}_{i_{1}} - {\bf Y}_{j_{1}}||~||{\bf X}_{i_{2}} - {\bf Y}_{j_{2}}||} - 1\right\}\right].
\end{eqnarray*}
Some straightforward algebra yields
\begin{eqnarray}
T_{WMW}^{(1)} &=& \frac{1}{d(m)_{2}(n)_{2}} \left\{\sum_{i_{1} \neq i_{2}} A_{i_{1},i_{2}} {\bf X}_{i_{1}}'{\bf X}_{i_{2}} - 2 \sum_{i, j} C_{i,j} {\bf X}_{i}'{\bf Y}_{j} + \sum_{j_{1} \neq j_{2}} B_{j_{1},j_{2}} {\bf Y}_{j_{1}}'{\bf Y}_{j_{2}} \right\} \nonumber \\
&=& \frac{1}{d(m)_{2}(n)_{2}} \left\{\sum_{i_{1} \neq i_{2}} [P_{i_{1}}P_{i_{2}}]^{-1}A_{i_{1},i_{2}} {\bf V}_{i_{1}}'{\bf V}_{i_{2}} - \ 2\sum_{i, j} [P_{i}Q_{j}]^{-1}C_{i,j} {\bf V}_{i}'(Q_{j}\mu+{\bf W}_{j}) \right. \label{eq3} \\
&& + \left.\sum_{j_{1} \neq j_{2}} [Q_{j_{1}}Q_{j_{2}}]^{-1}B_{j_{1},j_{2}} (Q_{j_{1}}\mu+{\bf W}_{j_{1}})'(Q_{j_{2}}\mu+{\bf W}_{j_{2}}) \right\},  \nonumber
\end{eqnarray}
where $A_{i_{1},i_{2}} = \sum_{j_{1} \neq j_{2}} (\sigma_{V}^{2}P_{i_{1}}^{-2} + \sigma_{W}^{2}Q_{j_{2}}^{-2})^{-1/2}(\sigma_{V}^{2}P_{i_{2}}^{-2} + \sigma_{W}^{2}Q_{j_{1}}^{-2})^{-1/2}$, $B_{j_{1},j_{2}} = \sum_{i_{1} \neq i_{2}} (\sigma_{V}^{2}P_{i_{1}}^{-2} + \sigma_{W}^{2}Q_{j_{2}}^{-2})^{-1/2}(\sigma_{V}^{2}P_{i_{2}}^{-2} + \sigma_{W}^{2}Q_{j_{1}}^{-2})^{-1/2}$, $C_{i_{1},j_{1}} = \sum_{i_{2} \neq i_{1}} \sum_{j_{2} \neq j_{1}} (\sigma_{V}^{2}P_{i_{1}}^{-2} + \sigma_{W}^{2}Q_{j_{2}}^{-2})^{-1/2}(\sigma_{V}^{2}P_{i_{2}}^{-2} + \sigma_{W}^{2}Q_{j_{1}}^{-2})^{-1/2}$.
Define
\begin{eqnarray*}
S_{1} &=& \frac{1}{(m)_{2}(n)_{2}} \sum_{i_{1} \neq i_{2}} \sum_{j_{1} \neq j_{2}} (\sigma_{V}^{2}P_{i_{1}}^{-2} + \sigma_{V}^{2}Q_{j_{2}}^{-2})^{-1/2}(\sigma_{V}^{2}P_{i_{2}}^{-2} + \sigma_{W}^{2}Q_{j_{1}}^{-2})^{-1/2}.
\end{eqnarray*}
It follows from the expression of $T_{WMW}^{(1)}$ in \eqref{eq3} that $E(dT_{WMW}^{(1)}|P_{i}, Q_{j}, 1 \leq i \leq m, 1 \leq j \leq n) = ||\mu||^{2}S_{1}$. Further, it can be shown that
\begin{eqnarray}
&& Var(dT_{WMW}^{(1)}|P_{i}, Q_{j}, 1 \leq i \leq m, 1 \leq j \leq n) \nonumber \\
&& = S_{2} + \frac{4}{[(m)_{2}(n)_{2}]^{2}} \sum_{i, j_{1} \neq j_{2}} \{\mu'\Sigma_{V}\mu\}P_{i}^{-2}C_{i,j_{1}}C_{i,j_{2}} \nonumber \\
&& \hspace{1cm} + \ \frac{1}{[(m)_{2}(n)_{2}]^{2}} \sum_{\substack{j_{1},j_{2},j_{3}\\all~distinct}} \left[ \{\mu'\Sigma_{W}\mu\}(P_{j_{1}}^{-2} + P_{j_{2}}^{-2})B_{j_{1},j_{2}}(2B_{j_{1},j_{2}} + B_{j_{1},j_{3}} + B_{j_{2},j_{3}}) \right], \nonumber \\
&& = S_{2} + \frac{1}{[(m)_{2}(n)_{2}]^{2}} \{L_{1}\mu'\Sigma_{V}\mu + L_{2}\mu'\Sigma_{W}\mu\},  \label{eq4}
\end{eqnarray}
where $L_{1} = 4\sum_{i, j_{1} \neq j_{2}} P_{i}^{-2}C_{i,j_{1}}C_{i,j_{2}}$ and $L_{2} = \sum (P_{j_{1}}^{-2} + P_{j_{2}}^{-2})B_{j_{1},j_{2}}(2B_{j_{1},j_{2}} + B_{j_{1},j_{3}} + B_{j_{2},j_{3}})$. Here, the latter summation is taken over distinct indices $j_{1}, j_{2}$ and $j_{3}$. Also,
\begin{eqnarray*}
S_{2} &=& \frac{1}{[(m)_{2}(n)_{2}]^{2}} \left\{2\sum_{i_{1} \neq i_{2}} [P_{i_{1}}P_{i_{2}}]^{-2}A_{i_{1},i_{2}}^{2} {\mbox{tr}}(\Sigma_{V}^{2}) \right. \nonumber \\
&& + \ \left. 2\sum_{j_{1} \neq j_{2}} [Q_{j_{1}}Q_{j_{2}}]^{-2}B_{j_{1},j_{2}}^{2} {\mbox{tr}}(\Sigma_{W}^{2}) + 4\sum_{i,j} [P_{i}Q_{j}]^{-2}C_{i,j}^{2} {\mbox{tr}}(\Sigma_{V}\Sigma_{W})\right\}, \nonumber \\
&=& \{L_{3}{\mbox{tr}}(\Sigma_{V}^{2}) + L_{4}{\mbox{tr}}(\Sigma_{W}^{2}) + 2L_{5}{\mbox{tr}}(\Sigma_{V}\Sigma_{W})\}/[(m)_{2}(n)_{2}]^{2},
\end{eqnarray*}
where $L_{3} = 2\sum_{i_{1} \neq i_{2}} [P_{i_{1}}P_{i_{2}}]^{-2}A_{i_{1},i_{2}}^{2}$, $L_{4} = 2\sum_{j_{1} \neq j_{2}} [Q_{j_{1}}Q_{j_{2}}]^{-2}B_{j_{1},j_{2}}^{2}$, and $L_{5} = 2\sum_{i,j} [P_{i}Q_{j}]^{-2}C_{i,j}^{2}$. Note that $(L_{1}\mu'\Sigma_{V}\mu + L_{2}\mu'\Sigma_{W}\mu) \leq \max\{L_{1},L_{2}\}\mu'(\Sigma_{V} + \Sigma_{W})\mu$. Also, $S_{2} \geq [(m)_{2}(n)_{2}]^{-2}\min\{L_{3},L_{4},$ $L_{5}\}[{\mbox{tr}}(\Sigma_{V}^{2}) + {\mbox{tr}}(\Sigma_{W}^{2}) + 2{\mbox{tr}}(\Sigma_{V}\Sigma_{W})] = [(m)_{2}(n)_{2}]^{-2}\min\{L_{3},L_{4},L_{5}\}{\mbox{tr}}[(\Sigma_{V} + \Sigma_{W})^{2}]$. These facts along with \eqref{eq4} and Assumption (C3) imply that $Var(dT_{WMW}^{(1)}|P_{i}, Q_{j}, 1 \leq i \leq m, 1 \leq j \leq n) = S_{2}(1+o(1))$ as $d \rightarrow \infty$. Now,
\begin{eqnarray}
&& (dT_{WMW}^{(1)} - ||\mu||^{2}S_{1})/S_{2}^{1/2} \nonumber \\
&& = \left[\frac{1}{(m)_{2}(n)_{2}} \sum_{i_{1} \neq i_{2}} \sum_{j_{1} \neq j_{2}}
\frac{({\bf X}_{i_{1}} - {\bf Y}_{j_{1}} + \mu)'({\bf X}_{i_{2}} - {\bf Y}_{j_{2}} + \mu)}{d(\sigma_{V}^{2}P_{i_{1}}^{-2} + \sigma_{W}^{2}Q_{j_{1}}^{-2})^{1/2}(\sigma_{V}^{2}P_{i_{2}}^{-2} + \sigma_{W}^{2}Q_{j_{2}}^{-2})^{1/2}} \right.\nonumber \\
&& \hspace{1cm}\left.- \ \frac{2}{(m)_{2}(n)_{2}} \sum_{i,j} C_{i,j}\mu'({\bf X}_{i} - {\bf Y}_{j} + \mu)\right]/S_{2}^{1/2} \nonumber \\
&& = (\widetilde{T}_{WMW}^{(1)} - \widetilde{T}_{WMW}^{(2)})/S_{2}^{1/2},  \label{eq5}
\end{eqnarray}
where
\begin{eqnarray*}
\widetilde{T}_{WMW}^{(1)} = \frac{1}{(m)_{2}(n)_{2}} \sum_{i_{1} \neq i_{2}} \sum_{j_{1} \neq j_{2}}
\frac{({\bf X}_{i_{1}} - {\bf Y}_{j_{1}} + \mu)'({\bf X}_{i_{2}} - {\bf Y}_{j_{2}} + \mu)}{d(\sigma_{V}^{2}P_{i_{1}}^{-2} + \sigma_{W}^{2}Q_{j_{1}}^{-2})^{1/2}(\sigma_{V}^{2}P_{i_{2}}^{-2} + \sigma_{W}^{2}Q_{j_{2}}^{-2})^{1/2}}
\end{eqnarray*}
and $\widetilde{T}_{WMW}^{(2)} = 2[(m)_{2}(n)_{2}]^{-1} \sum_{i,j} C_{i,j}\mu'({\bf X}_{i} - {\bf Y}_{j} + \mu)$. It can be shown that $E(\widetilde{T}_{WMW}^{(2)}|P_{i}, Q_{j}, 1 \leq i \leq m, 1 \leq j \leq n) = 0$ and
\begin{eqnarray*}
&& Var(\widetilde{T}_{WMW}^{(2)}|P_{i}, Q_{j}, 1 \leq i \leq m, 1 \leq j \leq n) = 4[(m)_{2}(n)_{2}]^{-2} \left\{\sum_{i,j_{1} \neq j_{2}} C_{i,j_{1}}C_{i,j_{2}}P_{i}^{-2}\mu'\Sigma_{V}\mu \right.\\
&& \hspace{3cm} + \ \left.\sum_{i,j} C_{i,j}^{2}\mu'(\Sigma_{V}/P_{i}^{2} + \Sigma_{W}/Q_{j}^{2})\mu + \sum_{i_{1} \neq i_{2}, j} C_{i_{1},j}C_{i_{2},j}Q_{j}^{-2}\mu'\Sigma_{W}\mu\right\}.
\end{eqnarray*}
So, using Assumption (C3) and arguments similar to those used earlier to show that $Var(dT_{WMW}^{(1)}|P_{i},$ $Q_{j}, 1 \leq i \leq m, 1 \leq j \leq n) = S_{2}(1+o(1))$ as $d \rightarrow \infty$, we get that $Var(\widetilde{T}_{WMW}^{(2)}|P_{i}, Q_{j}, 1 \leq i \leq m, 1 \leq j \leq n) = o(S_{2})$ as $d \rightarrow \infty$. Thus, Chebyshev's inequality implies that $\widetilde{T}_{WMW}^{(2)}/S_{2}^{1/2}$ converges to zero {\it in probability} as $d \rightarrow \infty$. \\
\indent Next note that
\begin{eqnarray*}
&& \widetilde{T}_{WMW}^{(1)} \\
&& = \frac{1}{(m)_{2}(n)_{2}} \sum_{i_{1} \neq i_{2}} \sum_{j_{1} \neq j_{2}} \left\{\frac{({\bf V}_{i_{1}}/P_{i_{1}} - {\bf W}_{j_{1}}/Q_{j_{1}})'({\bf V}_{i_{2}}/P_{i_{2}} - {\bf W}_{j_{2}}/Q_{j_{2}})}{d(\sigma_{V}^{2}P_{i_{1}}^{-2} + \sigma_{W}^{2}Q_{j_{1}}^{-2})^{1/2}(\sigma_{V}^{2}P_{i_{2}}^{-2} + \sigma_{W}^{2}Q_{j_{2}}^{-2})^{1/2}}\right\} \\
&& = \frac{1}{(m)_{2}(n)_{2}} \sum_{k=1}^{d} \sum_{i_{1} \neq i_{2}} \sum_{j_{1} \neq j_{2}} \left\{\frac{(Q_{j_{1}}V_{i_{1}k} - P_{i_{1}}W_{j_{1}k})(Q_{j_{2}}V_{i_{2}k} - P_{i_{2}}W_{j_{2}k})}{dP_{i_{1}}P_{i_{2}}Q_{j_{1}}Q_{j_{2}}(\sigma_{V}^{2}P_{i_{1}}^{-2} + \sigma_{W}^{2}Q_{j_{1}}^{-2})^{1/2}(\sigma_{V}^{2}P_{i_{2}}^{-2} + \sigma_{W}^{2}Q_{j_{2}}^{-2})^{1/2}}\right\}.
\end{eqnarray*}
It is easy to see that $E(\widetilde{T}_{WMW}^{(1)}|P_{i}, Q_{j}, 1 \leq i \leq m, 1 \leq j \leq n) = 0$. Further, from algebraic computations similar to those used earlier in deriving $Var(dT_{WMW}^{(1)}|P_{i}, Q_{j}, 1 \leq i \leq m, 1 \leq j \leq n)$, it can be shown that $Var(\widetilde{T}_{WMW}^{(1)}|P_{i}, Q_{j}, 1 \leq i \leq m, 1 \leq j \leq n) = S_{2}$. Thus, by Theorem 4.0.1 in \cite{LL96} and Assumption (C4), the conditional distribution of $\widetilde{T}_{WMW}^{(1)}/S_{2}^{1/2}$ given the $P_{i}$'s and the $Q_{j}$'s converges to a standard Gaussian distribution as $d \rightarrow \infty$. This fact along with \eqref{eq5} and the fact that conditionally on the $P_{i}$'s and the $Q_{j}$'s, $\widetilde{T}_{WMW}^{(2)}/S_{2}^{1/2}$ converges to zero {\it in probability}  as $d \rightarrow \infty$ yield
\begin{eqnarray}
\lim_{d \rightarrow \infty} P\{(dT_{WMW}^{(1)} - ||\mu||^{2}S_{1})/S_{2}^{1/2} \leq x\} = \Phi(x). \label{eq6}
\end{eqnarray}
\indent Next, let us write
\begin{eqnarray}
&& T_{WMW}^{(2)} \nonumber \\
&& = \frac{1}{(m)_{2}(n)_{2}} \sum_{i_{1} \neq i_{2}} \sum_{j_{1} \neq j_{2}} \left[
\frac{({\bf X}_{i_{1}} - {\bf Y}_{j_{1}})'({\bf X}_{i_{2}} - {\bf Y}_{j_{2}}) - ||\mu||^{2}}{(\sigma_{V}^{2}P_{i_{1}}^{-2} + \sigma_{W}^{2}Q_{j_{1}}^{-2})^{1/2}(\sigma_{V}^{2}P_{i_{2}}^{-2} + \sigma_{W}^{2}Q_{j_{2}}^{-2})^{1/2}} \times \right. \nonumber \\
&& \hspace{2cm}\left.\left\{\frac{d(\sigma_{V}^{2}P_{i_{1}}^{-2} + \sigma_{W}^{2}Q_{j_{1}}^{-2})^{1/2}(\sigma_{V}^{2}P_{i_{2}}^{-2} + \sigma_{W}^{2}Q_{j_{2}}^{-2})^{1/2}}{||{\bf X}_{i_{1}} - {\bf Y}_{j_{1}}||~||{\bf X}_{i_{2}} - {\bf Y}_{j_{2}}||} - 1\right\}\right] \nonumber \\
&& \hspace{0.5cm}+ \ \frac{||\mu||^{2}}{(m)_{2}(n)_{2}} \sum_{i_{1} \neq i_{2}} \sum_{j_{1} \neq j_{2}} \left[\frac{1}{(\sigma_{V}^{2}P_{i_{1}}^{-2} + \sigma_{W}^{2}Q_{j_{1}}^{-2})^{1/2}(\sigma_{V}^{2}P_{i_{2}}^{-2} + \sigma_{W}^{2}Q_{j_{2}}^{-2})^{1/2}} \times \right. \nonumber \\
&& \hspace{2cm}\left.\left\{\frac{d(\sigma_{V}^{2}P_{i_{1}}^{-2} + \sigma_{W}^{2}Q_{j_{1}}^{-2})^{1/2}(\sigma_{V}^{2}P_{i_{2}}^{-2} + \sigma_{W}^{2}Q_{j_{2}}^{-2})^{1/2}}{||{\bf X}_{i_{1}} - {\bf Y}_{j_{1}}||~||{\bf X}_{i_{2}} - {\bf Y}_{j_{2}}||} - 1\right\}\right] \nonumber \\
&& = \widetilde{T}_{WMW}^{(3)} + \widetilde{T}_{WMW}^{(4)}  \label{eq7},
\end{eqnarray}
where
\begin{eqnarray*}
\widetilde{T}_{WMW}^{(3)} = \frac{1}{(m)_{2}(n)_{2}} \sum_{i_{1} \neq i_{2}} \sum_{j_{1} \neq j_{2}} \left[\frac{({\bf X}_{i_{1}} - {\bf Y}_{j_{1}})'({\bf X}_{i_{2}} - {\bf Y}_{j_{2}}) - ||\mu||^{2}}{(\sigma_{V}^{2}P_{i_{1}}^{-2} + \sigma_{W}^{2}Q_{j_{1}}^{-2})^{1/2}(\sigma_{V}^{2}P_{i_{2}}^{-2} + \sigma_{W}^{2}Q_{j_{2}}^{-2})^{1/2}} \times \right. \nonumber \\
\hspace{2cm}\left.\left\{\frac{d(\sigma_{V}^{2}P_{i_{1}}^{-2} + \sigma_{W}^{2}Q_{j_{1}}^{-2})^{1/2}(\sigma_{V}^{2}P_{i_{2}}^{-2} + \sigma_{W}^{2}Q_{j_{2}}^{-2})^{1/2}}{||{\bf X}_{i_{1}} - {\bf Y}_{j_{1}}||~||{\bf X}_{i_{2}} - {\bf Y}_{j_{2}}||} - 1\right\}\right] \ \ \ \mbox{and} \\
\widetilde{T}_{WMW}^{(4)} = \frac{1}{(m)_{2}(n)_{2}} \sum_{i_{1} \neq i_{2}} \sum_{j_{1} \neq j_{2}} \left[\frac{||\mu||^{2}}{(\sigma_{V}^{2}P_{i_{1}}^{-2} + \sigma_{W}^{2}Q_{j_{1}}^{-2})^{1/2}(\sigma_{V}^{2}P_{i_{2}}^{-2} + \sigma_{W}^{2}Q_{j_{2}}^{-2})^{1/2}} \times \right. \nonumber \\
\hspace{2cm}\left.\left\{\frac{d(\sigma_{V}^{2}P_{i_{1}}^{-2} + \sigma_{W}^{2}Q_{j_{1}}^{-2})^{1/2}(\sigma_{V}^{2}P_{i_{2}}^{-2} + \sigma_{W}^{2}Q_{j_{2}}^{-2})^{1/2}}{||{\bf X}_{i_{1}} - {\bf Y}_{j_{1}}||~||{\bf X}_{i_{2}} - {\bf Y}_{j_{2}}||} - 1\right\}\right].
\end{eqnarray*}
As mentioned earlier, $S_{2} \geq [(m)_{2}(n)_{2}]^{-2}\min\{L_{3},L_{4},L_{5}\}{\mbox{tr}}[(\Sigma_{V} + \Sigma_{W})^{2}]$. Moreover, the stationarity of the sequences ${\cal X}$ and ${\cal Y}$ and the Cauchy-Schwarz inequality imply that ${\mbox{tr}}[(\Sigma_{V}+\Sigma_{W})^{2}] \geq d(\sigma_{V}^{2}+\sigma_{W}^{2})^{2}$. These facts along with \eqref{eq1-1} and Assumption (C3) imply that conditionally on the $P_{i}$'s and the $Q_{j}$'s, each term inside the double sum appearing in $\widetilde{T}_{WMW}^{(4)}$ above is $o_{P}(S_{2}^{1/2})$ as $d \rightarrow \infty$. So, $\widetilde{T}_{WMW}^{(4)}/S_{2}^{1/2}$ converges to zero {\it in probability} as $d \rightarrow \infty$. \\
\indent Next, fix any $i_{1} \neq i_{2}$ and $j_{1} \neq j_{2}$ and consider the corresponding term inside the double summation appearing in the expression of $\widetilde{T}_{WMW}^{(3)}$. It follows from \eqref{eq1-1} that $d(\sigma_{V}^{2}P_{i_{1}}^{-2} + \sigma_{W}^{2}Q_{j_{1}}^{-2})^{1/2}(\sigma_{V}^{2}P_{i_{2}}^{-2} + \sigma_{W}^{2}Q_{j_{2}}^{-2})^{1/2}/[||{\bf X}_{i_{1}} - {\bf Y}_{j_{1}}||~||{\bf X}_{i_{2}} - {\bf Y}_{j_{2}}||] - 1$ converges to zero {\it in probability} as $d \rightarrow \infty$. Also, note that
\begin{eqnarray}
&& ({\bf X}_{i_{1}} - {\bf Y}_{j_{1}})'({\bf X}_{i_{2}} - {\bf Y}_{j_{2}}) - ||\mu||^{2} = ({\bf X}_{i_{1}} - {\bf Y}_{j_{1}} + \mu)'({\bf X}_{i_{2}} - {\bf Y}_{j_{2}} + \mu) \label{eq8-new} \\
&& \hspace{4.8cm}- \ \mu'({\bf X}_{i_{1}} - {\bf Y}_{j_{1}} + \mu) - \mu'({\bf X}_{i_{2}} - {\bf Y}_{j_{2}} + \mu) \nonumber \\
&& \hspace{3cm}= \sum_{k=1}^{d} \left\{\frac{(Q_{j_{1}}V_{i_{1}k} - P_{i_{1}}W_{j_{1}k})(Q_{j_{2}}V_{i_{2}k} - P_{i_{2}}W_{j_{2}k})}{P_{i_{1}}P_{i_{2}}Q_{j_{1}}Q_{j_{2}}}\right\} \label{eq8} \\
&& - \ \mu'(Q_{j_{1}}{\bf V}_{i_{1}} - P_{i_{1}}{\bf W}_{j_{1}})\mu/(P_{i_{1}}Q_{j_{1}}) - \mu'(Q_{j_{2}}{\bf V}_{i_{2}} - P_{i_{2}}{\bf W}_{j_{2}})\mu/(P_{i_{2}}Q_{j_{2}}). \nonumber
\end{eqnarray}
It is easy to show that the conditional expectation of the first term in \eqref{eq8} given the $P_{i}$'s and the $Q_{j}$'s is zero, and its conditional variance is $v_{i_{1}i_{2}j_{1}j_{2}} = [P_{i_{1}}P_{i_{2}}]^{-2}{\mbox{tr}}(\Sigma_{V}^{2}) + [Q_{j_{1}}Q_{j_{2}}]^{-2}{\mbox{tr}}(\Sigma_{W}^{2}) + \{[P_{i_{1}}Q_{j_{2}}]^{-2} + [P_{i_{2}}Q_{j_{1}}]^{-2}\}{\mbox{tr}}(\Sigma_{V}\Sigma_{W})$. So, $v_{i_{1}i_{2}j_{1}j_{2}} = O({\mbox{tr}}[(\Sigma_{V} + \Sigma_{W})^{2}])$. Hence, using the fact that $S_{2} \geq [(m)_{2}(n)_{2}]^{-2}$ $\min\{L_{3},L_{4},L_{5}\}{\mbox{tr}}[(\Sigma_{V} + \Sigma_{W})^{2}]$ and Chebyshev's inequality, it follows that the first term in \eqref{eq8} after scaling by $S_{2}^{1/2}$ is bounded {\it in probability}, conditional on the $P_{i}$'s and the $Q_{j}$'s, as $d \rightarrow \infty$. Using Assumption (C3), Chebyshev's inequality and arguments similar to those used to prove the convergence {\it in probability} to zero of $\widetilde{T}_{WMW}^{(2)}$ earlier, we get that the second and the third terms in \eqref{eq8} after scaling by $S_{2}^{1/2}$ converge to zero {\it in probability} as $d \rightarrow \infty$. So, the left hand side of the equation \eqref{eq8-new} after scaling by $S_{2}^{1/2}$ is bounded {\it in probability}, conditional on the $P_{i}$'s and the $Q_{j}$'s, as $d \rightarrow \infty$. Thus, $\widetilde{T}_{WMW}^{(3)}/S_{2}^{1/2}$ converges to zero {\it in probability} as $d \rightarrow \infty$. This along with \eqref{eq7} and the fact that $\widetilde{T}_{WMW}^{(4)}/S_{2}^{1/2}$ converges to zero {\it in probability} as $d \rightarrow \infty$ together imply that $T_{WMW}^{(2)}/S_{2}^{1/2}$ converges to zero {\it in probability} as $d \rightarrow \infty$. Combining this fact with \eqref{eq6} and \eqref{eq2}, we get $\lim_{d \rightarrow \infty} P\{(dT_{WMW} - ||\mu||^{2}S_{1})/S_{2}^{1/2} \leq x|P_{i},Q_{j}, 1 \leq i \leq m, 1 \leq j \leq n\} = \Phi(x)$ for all $x \in \mathbb{R}$ and for each $m,n \geq 1$. Consequently,
\begin{eqnarray*}
\lim_{d \rightarrow \infty} P\{(dT_{WMW} - ||\mu||^{2}S_{1})/S_{2}^{1/2} \leq x\} = \Phi(x)
\end{eqnarray*}
for all $x \in \mathbb{R}$ and for each $m,n \geq 1$. \\
\indent We now derive the asymptotic distribution of $T_{CQ}^{(2)}$. As in the proof of Theorem \ref{thm-new000}, $T_{CQ}^{(2)} = T_{1} - T_{2}$. In the setup of the present theorem, $T_{1} = [(m)_{2}(n)_{2}]^{-1} \sum_{i_{1} \neq i_{2}} \sum_{j_{1} \neq j_{2}} ({\bf V}_{i_{1}}/P_{i_{1}} - {\bf W}_{j_{1}}/Q_{j_{1}})'({\bf V}_{i_{2}}/P_{i_{2}} - {\bf W}_{j_{2}}/Q_{j_{2}})'$ and $T_{2} = 2(mn)^{-1} \sum_{i, j} \mu'({\bf V}_{i}/P_{i} - {\bf W}_{j}/Q_{j})$. So, $E(T_{1}|P_{i},Q_{j}, 1 \leq i \leq m, 1 \leq j \leq n) = 0$. Further, from algebraic computations similar to those used to derive the variance of $T_{CQ}^{(2)}$ in the proof of Theorem \ref{thm-new000}, it follows that
\begin{eqnarray*}
Var(T_{1}|P_{i},Q_{j}, 1 \leq i \leq m, 1 \leq j \leq n) = \frac{1}{[(m)_{2}(n)_{2}]^{2}} \left\{2\sum_{i_{1} \neq i_{2}} [P_{i_{1}}P_{i_{2}}]^{-2}{\mbox{tr}(\Sigma_{V}^{2})} \right.\\
\hspace{2cm} + 2\sum_{j_{1} \neq j_{2}} [Q_{j_{1}}Q_{j_{2}}]^{-2}{\mbox{tr}(\Sigma_{W}^{2})} + \left.4\sum_{i, j} [P_{i}Q_{j}]^{-2}{\mbox{tr}(\Sigma_{V}\Sigma_{W})}\right\}.
\end{eqnarray*}
Define $S_{3} = Var(T_{1}|P_{i},Q_{j}, 1 \leq i \leq m, 1 \leq j \leq n)$. Also, $E(T_{2}|P_{i},Q_{j}, 1 \leq i \leq m, 1 \leq j \leq n) = 0$, and $Var(T_{2}|P_{i},Q_{j}, 1 \leq i \leq m, 1 \leq j \leq n) = o(S_{3})$ as $d \rightarrow \infty$ using the assumptions in the theorem. Thus, $T_{2}/S_{3}^{1/2}$ converges {\it in probability} to zero as $d \rightarrow \infty$. Further, using arguments similar to those used to prove the asymptotic Gaussianity of $\widetilde{T}_{WMW}^{(1)}$ above, it follows that the conditional distribution of $T_{1}/S_{3}^{1/2}$ given the $P_{i}$'s and the $Q_{j}$'s converges {\it weakly} to a standard Gaussian distribution as $d \rightarrow \infty$ for all $m, n \geq 1$. Combining these facts, we have 
\begin{eqnarray*}
\lim_{d \rightarrow \infty} P\{(T_{CQ}^{(2)} - ||\mu||^{2}/S_{3}^{1/2} \leq x\} = \Phi(x)
\end{eqnarray*}
for all $x \in \mathbb{R}$ and all $m,n \geq 1$. \vspace{0.05in}\\
(b) \indent Note that $S_{1}$ is a real valued $V$-statistic whose kernel $(\sigma_{V}^{2}P_{i_{1}}^{-2} + \sigma_{V}^{2}Q_{j_{2}}^{-2})^{-1/2}(\sigma_{V}^{2}P_{i_{2}}^{-2} + \sigma_{W}^{2}Q_{j_{1}}^{-2})^{-1/2}$ has finite expectation $\psi_{1} = E^{2}\{PQ/(\sigma_{V}^{2}Q^{2} + \sigma_{W}^{2}P^{2})^{1/2}\}$, by the assumption in the theorem. Thus, it follows that $S_{1}$ converges {\it almost surely} to $\psi_{1}$. Define $S_{21} = [(m)_{2}\{(n)_{2}\}^{2}]^{-1}L_{3}$, $S_{22} = [\{(m)_{2}\}^{2}(n)_{2}]^{-1}L_{4}$ and $S_{23} = [mn(m-1)^{2}(n-1)^{2}]^{-1}L_{5}$. Each of $S_{21}$, $S_{22}$ and $S_{23}$ is a real valued $V$-statistic whose kernel is bounded and thus has finite expectation. So, there exist $\psi_{21}$, $\psi_{22}$ and $\psi_{23}$ depending only on the distributions of $P$ and $Q$ such that $S_{21}$, $S_{22}$ and $S_{23}$ converge {\it almost surely} to $\psi_{21}$, $\psi_{22}$ and $\psi_{23}$, respectively. Here, $\psi_{21} = E^{2}\{Q_{1}Q_{2}/[(\sigma_{V}^{2}Q_{2}^{2} + \sigma_{W}^{2}P_{1}^{2})(\sigma_{V}^{2}Q_{2}^{2} + \sigma_{W}^{2}P_{1}^{2})]^{1/2}\}$, $\psi_{22} = E^{2}\{P_{1}P_{2}/[(\sigma_{V}^{2}Q_{1}^{2} + \sigma_{W}^{2}P_{1}^{2})(\sigma_{V}^{2}Q_{1}^{2} + \sigma_{W}^{2}P_{2}^{2})]^{1/2}\}$, and $\psi_{23} = [\psi_{21}\psi_{22}]^{1/2}$. Define $\psi_{2} = 2{\mbox{tr}}(\Sigma_{V}^{2})\psi_{21}/(m)_{2} + 2{\mbox{tr}}(\Sigma_{W}^{2})\psi_{22}/(n)_{2} + 4{\mbox{tr}}(\Sigma_{V}\Sigma_{W})\psi_{23}/(mn)$. Recall that $S_{2} = 2{\mbox{tr}}(\Sigma_{V}^{2})S_{21}/(m)_{2} + 2{\mbox{tr}}(\Sigma_{W}^{2})S_{22}/(n)_{2} + 4{\mbox{tr}}(\Sigma_{V}\Sigma_{W})S_{23}/(mn)$. Conditions (C1) and (C2) along with Theorem 2.1.5 in \cite{LL96} imply that both ${\cal V}$ and ${\cal W}$ possess continuous spectral densities. Now, the proof of Theorem 18.2.1 in \cite{IL71} implies that each of ${\mbox{tr}}(\Sigma_{V}^{2})$, ${\mbox{tr}}(\Sigma_{W}^{2})$ and ${\mbox{tr}}(\Sigma_{V}\Sigma_{W})$ equals a constant multiple of $d$ plus a remainder term, which is $o(d)$ as $d \rightarrow \infty$. Thus, for each fixed $m,n \geq 1$, there exist constants $A_{1}$, $A_{2}$ and $A_{3}$ such that with {\it probability one}
\begin{eqnarray}
\lim_{d \rightarrow \infty} \frac{\psi_{2}}{S_{2}} = \frac{2\psi_{21}A_{1}/(m)_{2} + 2\psi_{22}A_{2}/(n)_{2} + 4\psi_{23}A_{3}/(mn)}{2S_{21}A_{1}/(m)_{2} + 2S_{22}A_{2}/(n)_{2} + 4S_{23}A_{3}/(mn)} \label{eq13}
\end{eqnarray}
We denote the right hand side of \eqref{eq13} by $R_{m,n}$. Further, the assumption in the theorem and arguments preceding \eqref{eq13} imply that $||\mu||^{2}/\psi_{2}^{1/2}$ converges to a finite non-negative limit $b^{2}$ (say) as $d \rightarrow \infty$. Now,
\begin{eqnarray*}
&& \lim_{m,n \rightarrow \infty} \lim_{d \rightarrow \infty} P\left\{ \frac{dT_{WMW} - ||\mu||^{2}\psi_{1}}{\psi_{2}^{1/2}} \leq x\right\} \\
&=& \lim_{m,n \rightarrow \infty} \lim_{d \rightarrow \infty} P\left\{ \frac{dT_{WMW} - ||\mu||^{2}S_{1}}{S_{2}^{1/2}} \leq \frac{x\psi_{2}^{1/2}}{S_{2}^{1/2}} - \frac{||\mu||^{2}(S_{1}-\psi_{1})}{S_{2}^{1/2}}\right\} \\
&=& \lim_{m,n \rightarrow \infty} E\left[ \lim_{d \rightarrow \infty} P\left\{ \frac{dT_{WMW} - ||\mu||^{2}S_{1}}{S_{2}^{1/2}} \leq \frac{x\psi_{2}^{1/2}}{S_{2}^{1/2}} - \frac{||\mu||^{2}(S_{1}-\psi_{1})}{S_{2}^{1/2}} | P_{i}'s, Q_{j}'s\right\} \right] \\
&=& \lim_{m,n \rightarrow \infty} E\left[ \Phi\left(\lim_{d \rightarrow \infty} \frac{\psi_{2}^{1/2}}{S_{2}^{1/2}}\left\{x - (S_{1}-\psi_{1}) \lim_{d \rightarrow \infty}\frac{||\mu||^{2}}{\psi_{2}^{1/2}}\right\}\right) | P_{i}'s, Q_{j}'s\right]  \\
&=& E\left[ \lim_{m,n \rightarrow \infty} \Phi(R_{m,n}\{x - (S_{1}-\psi_{1})b^{2}\}) | P_{i}'s, Q_{j}'s\right] \ = \ \Phi(x),
\end{eqnarray*}
where the last equality above follows since $R_{m,n}$ converges to one and $S_{1} - \psi_{1}$ converges to zero  {\it almost surely} as $m,n \rightarrow \infty$. \\
\indent Note that $[(m)_{2}\{(n)_{2}\}^{2}]^{-1} \sum_{i_{1} \neq i_{2}} [P_{i_{1}}P_{i_{2}}]^{-2}$, $[(n)_{2}\{(m)_{2}\}^{2}]^{-1} \sum_{j_{1} \neq j_{2}} [Q_{j_{1}}Q_{j_{2}}]^{-2}$ and $[mn(m-1)^{2}(n-1)^{2}]^{-1} \sum_{i, j} [P_{i}Q_{j}]^{-2}$ appearing in the expression of $S_{3}$ converge to $E^{2}(P^{-2})$, $E^{2}(Q^{-2})$ and $E(P^{-2})E(Q^{-2})$, respectively, as $m,n \rightarrow \infty$. Also note that $\Sigma_{1} = Disp({\bf X}) = \Sigma_{V}E(P^{-2})$ and $\Sigma_{2} = Disp({\bf Y}) = \Sigma_{W}E(Q^{-2})$. So, arguing as in the case of $S_{2}$ above, we get that $S_{3}/\Gamma_{1}$ converges {\it in probability} to one as first $d \rightarrow \infty$ and then $m,n \rightarrow \infty$. Thus, it follows that $\lim_{m,n \rightarrow \infty} \lim_{d \rightarrow \infty} P\{(T_{CQ}^{(2)} - ||\mu||^{2})/\Gamma_{1}^{1/2} \leq x\} = \Phi(x)$ for all $x \in \mathbb{R}$.
\end{proof}

\begin{proof}[Proof of Theorem \ref{thm-new100}]
Since ${\bf Y}$ is distributed as ${\bf X} + \mu$, we have $\psi_{1} = \sigma_{V}^{-2}E^{2}\{PQ/(P^{2}+Q^{2})^{1/2}\}$ and $\psi_{2} = [\sigma_{V}^{2}E(P^{-2})]^{-2}E^{2}\{Q_{1}Q_{2}/[(P_{1}^{2}+Q_{1}^{2})^{1/2}(P_{1}^{2}+Q_{2}^{2})^{1/2}]\}\Gamma_{1}$. Here, $\psi_{1}$ and $\psi_{2}$ are as in the proof of Theorem \ref{thm-new10}. Since $\lim_{m,n \rightarrow \infty} \lim_{d \rightarrow \infty} ||\mu||^{2}/\Gamma_{2}^{1/2} = c$ for some $c \in (0,\infty)$, we have $\lim_{m,n \rightarrow \infty} \lim_{d \rightarrow \infty} \beta_{T_{CQ}^{(2)}}(\mu) = \Phi(-\zeta_{\alpha} + c)$, and
\begin{eqnarray*}
\lim_{m,n \rightarrow \infty} \lim_{d \rightarrow \infty} \beta_{T_{WMW}}(\mu) = \Phi\left(-\zeta_{\alpha} + \frac{cE(P^{-2})E^{2}\{PQ/(P^{2}+Q^{2})^{1/2}\}}{E\{Q_{1}Q_{2}/[(P_{1}^{2}+Q_{1}^{2})^{1/2}(P_{1}^{2}+Q_{2}^{2})^{1/2}]\}}\right).
\end{eqnarray*}
Now, $E^{2}\{Q_{1}Q_{2}/[(P_{1}^{2} + Q_{1}^{2})^{1/2}(P_{1}^{2} + Q_{2}^{2})^{1/2}]\} = E[E^{2}\{Q_{1}/(P_{1}^{2} + Q_{1}^{2})^{1/2}|P_{1}\}] < E[E\{Q_{1}^{2}/(P_{1}^{2} + Q_{1}^{2})|P_{1}\}] = E\{Q_{1}^{2}/(P_{1}^{2} + Q_{1}^{2})\} = 1/2$. Here, the inequality can be obtained using Jensen's inequality. Further, $E^{2}\{PQ/(P^{2} + Q^{2})^{1/2}\} > E^{-2}\{(P^{2} + Q^{2})^{1/2}/PQ\} > E^{-1}\{(P^{2} + Q^{2})/P^{2}Q^{2}\} = 1/\{E(P^{-2}) + E(Q^{-2})\} = [2E(P^{-2})]^{-1}$. Here, the inequalities follow from Cauchy-Schwarz inequality. Combining the previous two inequalities, we get $\lim_{m,n \rightarrow \infty} \lim_{d \rightarrow \infty} \beta_{T_{WMW}}(\mu) > \lim_{m,n \rightarrow \infty} \lim_{d \rightarrow \infty} \beta_{T_{CQ}^{(2)}}(\mu)$.
\end{proof}

\begin{proof}[Proof of Theorem \ref{thm-new3}]
(a) Let us consider the conditional distribution of $T_{SR}$ given $P_{1},P_{2},\ldots,P_{n}$. By definition,
\begin{eqnarray*}
 T_{SR} = \frac{1}{(n)_{4}} \sum_{\substack{i_{1}, i_{2}, i_{3}, i_{4}\\all~distinct}}
\frac{(P_{i_{2}}{\bf V}_{i_{1}} + P_{i_{1}}{\bf V}_{i_{2}} +
2{\mu}P_{i_{1}}P_{i_{2}})'(P_{i_{4}}{\bf V}_{i_{3}} + P_{i_{3}}{\bf V}_{i_{4}} +
2{\mu}P_{i_{3}}P_{i_{4}})}{||P_{i_{2}}{\bf V}_{i_{1}} + P_{i_{1}}{\bf V}_{i_{2}} +
2{\mu}P_{i_{1}}P_{i_{2}}||~||P_{i_{4}}{\bf V}_{i_{3}} + P_{i_{3}}{\bf V}_{i_{4}} +
2{\mu}P_{i_{3}}P_{i_{4}}||}.
\end{eqnarray*}
Consider the event $F = \{d^{-1}||P_{2}{\bf V}_{1} + P_{1}{\bf V}_{2} + 2{\mu}P_{1}P_{2}||^{2} - \sigma_{V}^{2}(P_{1}^{2} + P_{2}^{2}) = o(d^{-1/2 + \epsilon}) ~\mbox{as}~d \rightarrow \infty\}$. From Theorem 8.2.2 in \citet{LL96} and the assumptions in the theorem, it follows that for any given $\epsilon \in (0,1/2)$, $Pr(F | P_{1}, P_{2}) = 1$ for {\it almost every} $P_{1}$ and $P_{2}$. Let us rewrite $T_{SR}$ as
\begin{eqnarray}
&& T_{SR} \nonumber \\
&& = \frac{1}{(n)_{4}} \sum_{i_{1} \neq i_{2} \neq i_{3} \neq i_{4}}
\frac{(P_{i_{2}}{\bf V}_{i_{1}} + P_{i_{1}}{\bf V}_{i_{2}} +
2{\mu}P_{i_{1}}P_{i_{2}})'(P_{i_{4}}{\bf V}_{i_{3}} + P_{i_{3}}{\bf V}_{i_{4}} +
2{\mu}P_{i_{3}}P_{i_{4}})}{d\sigma_{V}^{2}(P_{i_{1}}^{2} + P_{i_{2}}^{2})^{1/2}(P_{i_{3}}^{2} +
P_{i_{4}}^{2})^{1/2}} \nonumber \\
&& + \frac{1}{(n)_{4}} \sum_{i_{1} \neq i_{2} \neq i_{3} \neq i_{4}} \left[
\frac{(P_{i_{2}}{\bf V}_{i_{1}} + P_{i_{1}}{\bf V}_{i_{2}} +
2{\mu}P_{i_{1}}P_{i_{2}})'(P_{i_{4}}{\bf V}_{i_{3}} + P_{i_{3}}{\bf V}_{i_{4}} +
2{\mu}P_{i_{3}}P_{i_{4}})}{d\sigma_{V}^{2}(P_{i_{1}}^{2} + P_{i_{2}}^{2})^{1/2}(P_{i_{3}}^{2} +
P_{i_{4}}^{2})^{1/2}}\right. \nonumber \\
&& \times \left.\left\{\frac{d\sigma_{V}^{2}(P_{i_{1}}^{2} + P_{i_{2}}^{2})^{1/2}(P_{i_{3}}^{2}
+ P_{i_{4}}^{2})^{1/2}}{||P_{i_{2}}{\bf V}_{i_{1}} + P_{i_{1}}{\bf V}_{i_{2}} +
2{\mu}P_{i_{1}}P_{i_{2}}||~||P_{i_{4}}{\bf V}_{i_{3}} + P_{i_{3}}{\bf V}_{i_{4}} +
2{\mu}P_{i_{3}}P_{i_{4}}||} - 1\right\}\right] \nonumber \\
&& \hspace{2cm} = (T_{SR}^{(1)} + T_{SR}^{(2)})/d, \label{eq4.1.1}
\end{eqnarray}
where
\begin{eqnarray*}
T_{SR}^{(1)} = \frac{1}{(n)_{4}} \sum_{i_{1} \neq i_{2} \neq i_{3} \neq i_{4}}
\frac{(P_{i_{2}}{\bf V}_{i_{1}} + P_{i_{1}}{\bf V}_{i_{2}} +
2{\mu}P_{i_{1}}P_{i_{2}})'(P_{i_{4}}{\bf V}_{i_{3}} + P_{i_{3}}{\bf V}_{i_{4}} +
2{\mu}P_{i_{3}}P_{i_{4}})}{\sigma_{V}^{2}(P_{i_{1}}^{2} + P_{i_{2}}^{2})^{1/2}(P_{i_{3}}^{2} +
P_{i_{4}}^{2})^{1/2}}
\end{eqnarray*}
and
\begin{eqnarray*}
T_{SR}^{(2)} = \frac{1}{(n)_{4}} \sum_{i_{1} \neq i_{2} \neq i_{3} \neq i_{4}} \left[
\frac{(P_{i_{2}}{\bf V}_{i_{1}} + P_{i_{1}}{\bf V}_{i_{2}} +
2{\mu}P_{i_{1}}P_{i_{2}})'(P_{i_{4}}{\bf V}_{i_{3}} + P_{i_{3}}{\bf V}_{i_{4}} +
2{\mu}P_{i_{3}}P_{i_{4}})}{\sigma_{V}^{2}(P_{i_{1}}^{2} + P_{i_{2}}^{2})^{1/2}(P_{i_{3}}^{2} +
P_{i_{4}}^{2})^{1/2}}\right. \\
\times \left.\left\{\frac{d\sigma_{V}^{2}(P_{i_{1}}^{2} + P_{i_{2}}^{2})^{1/2}(P_{i_{3}}^{2}
+ P_{i_{4}}^{2})^{1/2}}{||P_{i_{2}}{\bf V}_{i_{1}} + P_{i_{1}}{\bf V}_{i_{2}} +
2{\mu}P_{i_{1}}P_{i_{2}}||~||P_{i_{4}}{\bf V}_{i_{3}} + P_{i_{3}}{\bf V}_{i_{4}} +
2{\mu}P_{i_{3}}P_{i_{4}}||} - 1\right\}\right].
\end{eqnarray*}
Recall that a similar decomposition of $T_{SR}$ was obtained in the proof of Theorem \ref{thm-new2}. The proof of the asymptotic Gaussianity of $T_{SR}$ follows from the ideas used to prove the asymptotic Gaussianity of $T_{WMW}$ in Theorem \ref{thm-new10}, and the details are provided in Appendix -- II. $Z_{2}$ and $Z_{3}$ appearing in the asymptotic Gaussian distribution of $T_{SR}$ are given by $Z_{2} = 2[(n)_{4}\sigma_{V}^{2}]^{-1} \sum_{i_{1} \neq i_{2}} \widetilde{U}_{i_{1},i_{2}}P_{i_{1}}P_{i_{2}}$ and $Z_{3} = 8{\mbox{tr}}(\Sigma_{V}^{2})[(n)_{4}\sigma_{V}^{2}]^{-2}\sum_{i_{1} \neq i_{2}} \widetilde{U}_{i_{1},i_{2}}^{2}$, where $\widetilde{U}_{i_{1},i_{2}} = \sum_{i_{3} \neq i_{4} \neq i_{1} \neq i_{2}} P_{i_{3}}P_{i_{4}}/[(P_{i_{1}}^{2} + P_{i_{3}}^{2})^{1/2}(P_{i_{2}}^{2} + P_{i_{4}}^{2})^{1/2}]$. \\
\indent The proof of the asymptotic Gaussianity of $T_{S}$ will follow from arguments similar to those used to prove the asymptotic Gaussianity of $T_{SR}$, and we skip the details. $Z_{1}$ and $\Gamma_{3}$ in the asymptotic distribution of $T_{S}$ are given by $Z_{1} = [(n)_{2}\sigma_{V}^{2}]^{-1} \sum_{i_{1} \neq i_{2}} P_{i_{1}}P_{i_{2}}$ and $\Gamma_{3} = 2{\mbox{tr}(\Sigma_{V}^{2})}/[(n)_{2}\sigma_{V}^{4}]$. \\
\indent The proof of the asymptotic Gaussianity of $T_{CQ}^{(1)}$ is also provided in Appendix -- II, and $Z_{4}$ appearing in its asymptotic distribution is given by $Z_{4} = 2{\mbox{tr}(\Sigma_{V}^{2})}[(n)_{2}]^{-2}\sum_{i_{1} \neq i_{2}} [P_{i_{1}}P_{i_{2}}]^{-2}$. \vspace{0.05in}\\
(b) Observe that $Z_{1}$, $Z_{2}$, $(n)_{4}Z_{3}$ and $(n)_{2}Z_{4}$ are real-valued $V$-statistics, whose kernels have finite expectations by the assumption in part (b) of the theorem. So, they converge {\it almost surely} as $n \rightarrow \infty$. The corresponding limits are $\theta_{1} = E^{2}(P_{1})/\sigma_{V}^{2}$, $\theta_{2} = E^{2}\{P_{1}P_{2}/(P_{1}^{2} + P_{2}^{2})^{1/2}\}/\sigma_{V}^{2}$, $\theta_{3} = {\mbox{tr}(\Sigma_{V}^{2})}E^{2}\{P_{2}P_{3}/(P_{1}^{2} + P_{2}^{2})^{1/2}(P_{1}^{2} + P_{3}^{2})^{1/2}\}/\sigma_{V}^{4}$ and $\theta_{4} = 2{\mbox{tr}(\Sigma_{V}^{2})}E^{2}(P_{1}^{-2})$. Note that since $E(P^{-2})$ is finite, we have $\Sigma = Disp({\bf X}) = \Sigma_{V}E(P^{-2})$ and $\sigma^{2} = Var(X_{1}) = \sigma_{V}^{2}E(P^{-2})$. So, $\theta_{4} = 2{\mbox{tr}(\Sigma^{2})}$. Arguments similar to those used in the proof of part (b) of Theorem \ref{thm-new10} complete the proof of part (b) of the present theorem. \vspace{0.05in}\\
(c) Suppose that $\lim_{n \rightarrow \infty} \lim_{d \rightarrow \infty} ||\mu||^{2}/\Gamma_{2}^{1/2} = c$ for some $c \in (0,\infty)$. Then, \begin{eqnarray*}
&& \lim_{n \rightarrow \infty} \lim_{d \rightarrow \infty} \beta_{T_{S}}(\mu) = \Phi(-\zeta_{\alpha} + cE^{2}(P_{1})E(P_{1}^{-2})), \\
&& \lim_{n \rightarrow \infty} \lim_{d \rightarrow \infty} \beta_{T_{SR}}(\mu) = \Phi(-\zeta_{\alpha} + \frac{cE^{2}\{P_{1}P_{2}/(P_{1}^{2} + P_{2}^{2})^{1/2}\}E(P_{1}^{-2})}{E\{P_{2}P_{3}/[(P_{1}^{2} + P_{2}^{2})^{1/2}(P_{1}^{2} + P_{3}^{2})^{1/2}]\}}), \\
&& \lim_{n \rightarrow \infty} \lim_{d \rightarrow \infty} \beta_{T_{CQ}^{(1)}}(\mu) = \Phi(-\zeta_{\alpha} + c).
\end{eqnarray*}
Now, from Jensen's inequality, we have $E^{2}(P_{1}) > E^{-2}(P_{1}^{-1}) > E^{-1}(P_{1}^{-2})$, which implies that $E^{2}(P_{1})E(P_{1}^{-2}) > 1$. Thus, $\lim_{n \rightarrow \infty} \lim_{d \rightarrow \infty} \beta_{T_{S}}(\mu) > \lim_{n \rightarrow \infty} \lim_{d \rightarrow \infty} \beta_{T_{CQ}^{(1)}}(\mu)$. The proof of the other part of the theorem is similar to the proof of Theorem \ref{thm-new100}.
\end{proof}

\section*{Appendix -- II}
\subsection*{Additional mathematical details related to the proof of part (a) of Theorem \ref{thm-new3}}
\indent Here, we provide more details related to the derivations of the asymptotic distributions of $T_{SR}$, $T_{S}$ and $T_{CQ}^{(1)}$ under the assumptions of Theorem \ref{thm-new3}. Recall that $T_{SR} = (T_{SR}^{(1)} + T_{SR}^{(2)})/d$, where $T_{SR}^{(1)}$ and $T_{SR}^{(2)}$ are defined in the proof of Theorem \ref{thm-new3} in Appendix -- I. Define $\widetilde{U}_{i_{1},i_{2}} = \sum_{i_{3} \neq i_{4} \neq i_{1} \neq i_{2}} P_{i_{3}}P_{i_{4}}/[(P_{i_{1}}^{2} + P_{i_{3}}^{2})^{1/2}(P_{i_{2}}^{2} + P_{i_{4}}^{2})^{1/2}]$. Then, by the definition of $T_{SR}^{(1)}$, we have
\begin{eqnarray*}
T_{SR}^{(1)} &=& \frac{4}{(n)_{4}\sigma_{V}^{2}} \sum_{i_{1} \neq i_{2}} \left\{\widetilde{U}_{i_{1},i_{2}}{\bf W}_{i_{1}}'{\bf W}_{i_{2}} + 2\widetilde{U}_{i_{1},i_{2}}P_{i_{2}}\mu'{\bf W}_{i_{1}} + ||\mu||^{2}\widetilde{U}_{i_{1},i_{2}}P_{i_{1}}P_{i_{2}}\right\}.
\end{eqnarray*}
It follows easily that $E(T_{SR}^{(1)}|P_{i}, 1 \leq i \leq n) = 4||\mu||^{2}[(n)_{4}\sigma_{V}^{2}]^{-1} \sum_{i_{1} \neq i_{2}} \widetilde{U}_{i_{1},i_{2}}P_{i_{1}}P_{i_{2}}$. Set $Z_{2} = 2[(n)_{4}\sigma_{V}^{2}]^{-1} \sum_{i_{1} \neq i_{2}} \widetilde{U}_{i_{1},i_{2}}P_{i_{1}}P_{i_{2}}$. Further, it can be shown using the assumptions in the theorem that $Var(T_{SR}^{(1)}|P_{i}, 1 \leq i \leq n) = 32{\mbox{tr}}(\Sigma_{V}^{2})[(n)_{4}\sigma_{V}^{2}]^{-2}$ $\sum_{i_{1} \neq i_{2}} \widetilde{U}_{i_{1},i_{2}}^{2}(1 + o(1))$ as $d \rightarrow \infty$. Let $Z_{3} = 8{\mbox{tr}}(\Sigma_{V}^{2})[(n)_{4}\sigma_{V}^{2}]^{-2}\sum_{i_{1} \neq i_{2}} \widetilde{U}_{i_{1},i_{2}}^{2}$. So, using arguments similar to those used to prove the asymptotic Gaussianity of $T_{WMW}^{(1)}$ in the proof of Theorem \ref{thm-new10}, it follows that $(T_{SR}^{(1)} - 2||\mu||^{2}Z_{2})/(2Z_{3}^{1/2})$ converges {\it weakly} to a standard Gaussian distribution as $d \rightarrow \infty$ for each $n \geq 1$. Moreover, using arguments similar to those used to prove the convergence {\it in probability} of $T_{WMW}^{(2)}$ in the proof of Theorem \ref{thm-new10}, it follows that $T_{SR}^{(2)}/Z_{3}^{1/2}$ converges to zero {\it in probability} as $d \rightarrow \infty$ for each $n \geq 1$. This fact along with the equation $T_{SR} = (T_{SR}^{(1)} + T_{SR}^{(2)})/d$ and the asymptotic Gaussianity of $T_{SR}^{(1)}$ yields
\begin{eqnarray*}
 \lim_{d \rightarrow \infty} P\{(dT_{SR} - 2||\mu||^{2}Z_{2})/(2Z_{3}^{1/2}) \leq x\} = \Phi(x)
\end{eqnarray*}
for all $x \in \mathbb{R}$ and each $n \geq 1$. Here, $\Phi$ is the standard Gaussian cumulative distribution function. \\
\indent Using very similar arguments as above, we get that
\begin{eqnarray*}
\lim_{d \rightarrow \infty} P\{(dT_{S} - ||\mu||^{2}Z_{1})/\Gamma_{3}^{1/2} \leq x\} = \Phi(x)
\end{eqnarray*}
for all $x \in \mathbb{R}$ and each $n \geq 1$, where $Z_{1} = [(n)_{2}\sigma_{V}^{2}]^{-1} \sum_{i_{1} \neq i_{2}} P_{i_{1}}P_{i_{2}}$, and $\Gamma_{3} = 2{\mbox{tr}(\Sigma_{V}^{2})}/[(n)_{2}\sigma_{V}^{4}]$. \\
\indent Next, consider the conditional distribution of $T_{CQ}^{(1)}$ given the $P_{i}$'s, and note that
\begin{eqnarray*}
T_{CQ}^{(1)} &=& \frac{1}{(n)_{2}} \sum_{i_{1} \neq i_{2}} \frac{({\bf V}_{i_{1}} + {\mu}P_{i_{1}})'({\bf V}_{i_{2}} + {\mu}P_{i_{2}})}{P_{i_{1}}P_{i_{2}}} \\
&=& \frac{1}{(n)_{2}} \sum_{i_{1} \neq i_{2}} \frac{{\bf V}_{i_{1}}'{\bf V}_{i_{2}}}{P_{i_{1}}P_{i_{2}}} + \frac{2}{n} \sum_{i} \frac{\mu'{\bf V}_{i}}{P_{i}} + ||\mu||^{2}.
\end{eqnarray*}
So, $E(T_{CQ}^{(1)}|P_{i}, 1 \leq i \leq n) = ||\mu||^{2}$, and $Var(T_{CQ}^{(1)}|P_{i}, 1 \leq i \leq n) = Z_{4}(1+o(1))$ as $d \rightarrow \infty$, where $Z_{4} = 2{\mbox{tr}(\Sigma_{V}^{2})}[(n)_{2}]^{-2}\sum_{i_{1} \neq i_{2}} [P_{i_{1}}P_{i_{2}}]^{-2}$. Using the assumptions in the theorem, it follows that conditional on the $P_{i}$'s, $\sum_{i} \mu'{\bf V}_{i}/P_{i} = o_{P}(Z_{4}^{1/2})$ as $d \rightarrow \infty$. Thus, we get
\begin{eqnarray*}
\lim_{d \rightarrow \infty} P\{(T_{CQ}^{(1)} - ||\mu||^{2})/Z_{4}^{1/2} \leq x\} = \Phi(x)
\end{eqnarray*}
for all $x \in \mathbb{R}$ and each $n \geq 1$.

\section*{Appendix -- III}
\subsection*{Detailed results of the simulation study done in Section \ref{4}}
\indent Here, we present the results on the sizes and the powers of the tests based on $T_{SKK}$ \citep{SKK13} and $T_{GCBL}$ \citep{GCBL14} discussed in Section \ref{4}. We also present the sizes and the powers of the test in \cite{CLX14} for which the test statistic is denoted by $T_{CLX}$. Table \ref{tab1} reports the sizes of these tests implemented using the asymptotic approximations given in their original papers under the models considered in subsections \ref{2.1} and \ref{3.1} of our paper. We also report the sizes of the tests implemented using the permutation distributions of these test statistics.

\begin{table*}[ht!]
\centering
\caption{Sizes of the tests based on $T_{SKK}$, $T_{GCBL}$ and $T_{CLX}$ under some simulated models}
\label{tab1}
\begin{tabular}{ccccc}
\hline
   &  & $AR(1)$ with & $AR(1)$ with & spherical $t(5)$ \\
   &  & Gaussian innovation & $t(5)$ innovation & distribution \\
Test & $d$ &   &   &  \\
\hline
 & 100 & 0.06 & 0.064 & 0.011 \\
 & 200 & 0.068 & 0.06 & 0.001 \\
$T_{SKK}$--original & 400 & 0.071 & 0.072 & 0 \\
 & 800 & 0.089 & 0.089 & 0 \\
 & 1600 & 0.101 & 0.089 & 0 \\
\hline
 & 100 & 0.045 & 0.039 & 0.043 \\
 & 200 & 0.047 & 0.048 & 0.044 \\
$T_{SKK}$--permutation & 400 & 0.054 & 0.043 & 0.039 \\
 & 800 & 0.048 & 0.052 & 0.049 \\
 & 1600 & 0.042 & 0.054 & 0.051 \\
\hline
 & 100 & 0.077 & 0.071 & 0.137 \\
 & 200 & 0.075 & 0.078 & 0.148 \\
$T_{GCBL}$-original & 400 & 0.086 & 0.081 & 0.141 \\
 & 800 & 0.125 & 0.134 & 0.152 \\
 & 1600 & 0.164 & 0.152 & 0.185 \\
\hline
 & 100 & 0.042 & 0.048 & 0.046 \\
 & 200 & 0.051 & 0.042 & 0.044 \\
$T_{GCBL}$--permutation & 400 & 0.05 & 0.056 & 0.038 \\
 & 800 & 0.046 & 0.047 & 0.042 \\
 & 1600 & 0.055 & 0.047 & 0.039 \\
\hline
 & 100 & 0.082 & 0.075 & 0.076 \\
 & 200 & 0.101 & 0.114 & 0.093 \\
$T_{CLX}$--original & 400 & 0.136 & 0.147 & 0.105 \\
 & 800 & 0.167 & 0.184 & 0.131 \\
\hline
\end{tabular}
\end{table*}
Note that we could not implement the test based on $T_{CLX}$ using its permutation distribution because the test procedure uses a computationally intensive optimization. For the same reason, we could not implement this test for $d = 1600$ under any of the above models using the asymptotic distribution given in \cite{CLX14}. Recall that we have discussed the sizes of the tests based on $T_{WMW}$ and $T_{CQ}^{(2)}$ for the above models in detail in subsections \ref{2.1} and \ref{3.1}.

\begin{figure}[ht!]
\begin{center}
\includegraphics[width=5in,height=1.7in]{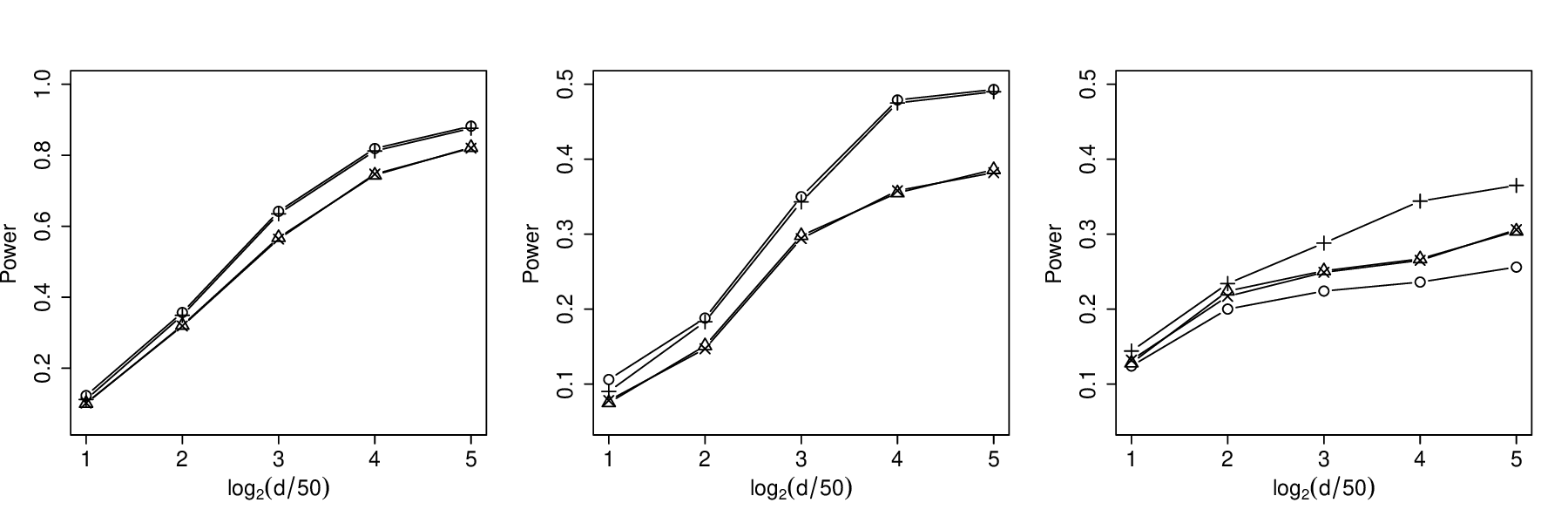}
\end{center}
\caption{Powers of the tests at nominal $5\%$ level based on $T_{WMW}$ (- + - curves), $T_{CQ}^{(2)}$ (- $\circ$ - curves), $T_{SKK}$ (- $\times$ - curves) and $T_{GCBL}$ (- $\triangle$ - curves) for the $AR(1)$ model with Gaussian innovation (left panel), the $AR(1)$ model with $t(5)$ innovation (middle panel) and the spherical $t(5)$ distribution (right panel). \label{f1}}
\end{figure}
\indent In Figure \ref{f1}, we give the plots of the empirical powers of the tests based on $T_{SKK}$ and $T_{GCBL}$, when they are implemented using their permutation distributions. Each plot in Figure \ref{f1} also includes the empirical powers of the tests based on $T_{WMW}$ and $T_{CQ}^{(2)}$. The power curves for these two tests are so close that they are overlaid on each other in the left and the middle plots. Similarly, the power curves corresponding to the tests based on $T_{SKK}$ and $T_{GCBL}$ are overlaid on each other in all the plots.

\begin{figure}[ht!]
\begin{center}
\includegraphics[width=5in,height=1.7in]{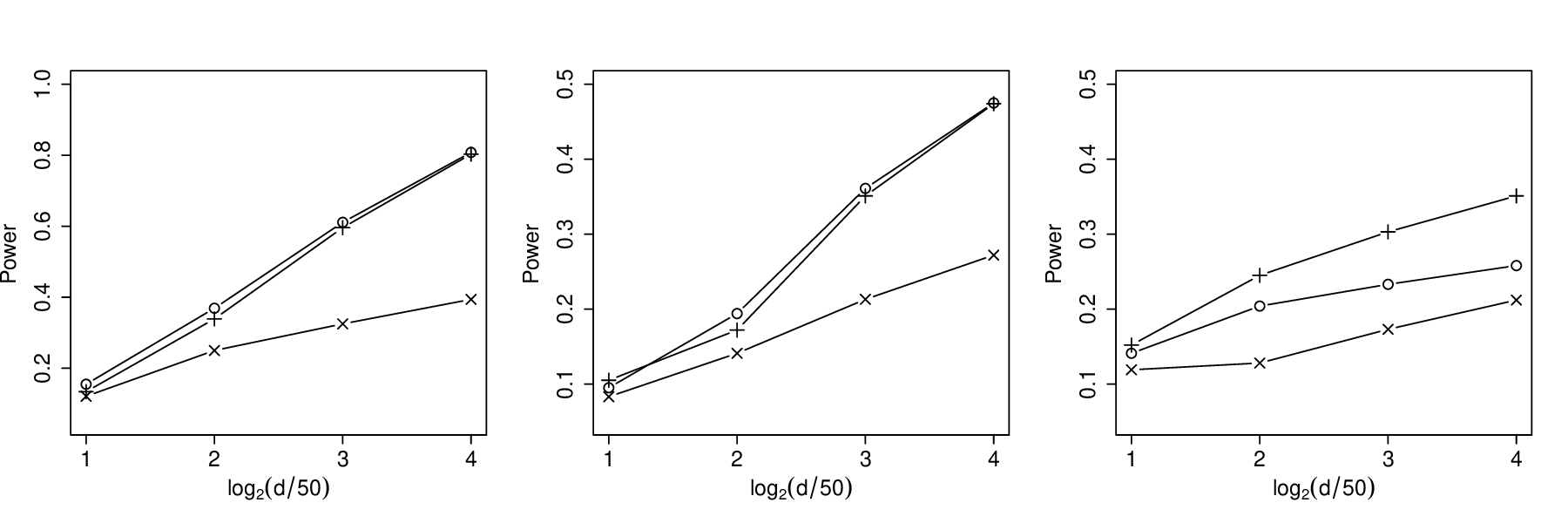}
\end{center}
\caption{Powers of the tests at nominal $5\%$ level based on $T_{WMW}$ (- + - curves), $T_{CQ}^{(2)}$ (- $\circ$ - curves) and $T_{CLX}$ (- $\times$ - curves) for the $AR(1)$ model with Gaussian innovation (left panel), the $AR(1)$ model with $t(5)$ innovation (middle panel) and the spherical $t(5)$ distribution (right panel). \label{f2}}
\end{figure}
\indent Figure \ref{f2} gives the plots of the empirical powers of the tests based on $T_{WMW}$, $T_{CQ}^{(2)}$ and $T_{CLX}$, when the mean shifts in the models considered in subsections \ref{2.1} and \ref{3.1} are distributed equally among all the coordinates. Once again, the power curves corresponding to the tests based on $T_{WMW}$ and $T_{CQ}^{(2)}$ are sufficiently close making the curves overlaid on each other in the left and the middle plots.

\indent In Figure \ref{f3}, we give the sizes and the powers of the tests based on $T_{WMW}$ and $T_{CQ}^{(2)}$ for the multivariate Gaussian distribution with dispersion matrix $(1-\beta)I_{d} + \beta{\bf 1}_{d}{\bf 1}_{d}^{'}$ with $\beta = 0.7$ considered in Section \ref{4}.
\begin{figure}[ht!]
\begin{center}
\includegraphics[width=4.8in,height=1.7in]{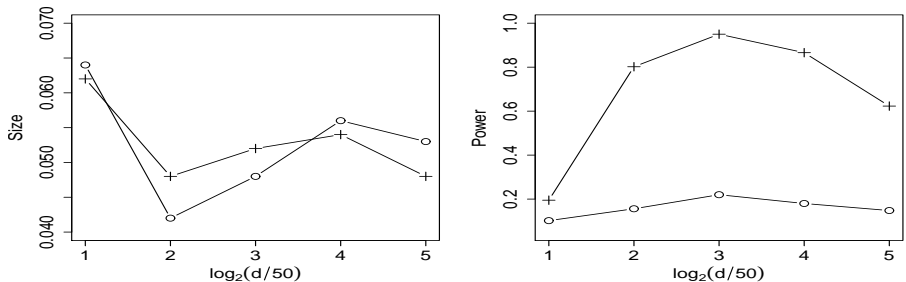}
\end{center}
\caption{Powers of the tests at nominal $5\%$ level based on $T_{WMW}$ (- + - curves) and $T_{CQ}^{(2)}$ (- $\circ$ - curves) for the multivariate Gaussian distribution with dispersion matrix $(1-\beta)I_{d} + \beta{\bf 1}_{d}{\bf 1}_{d}^{'}$ with $\beta = 0.7$. \label{f3}}
\end{figure}

\bibliographystyle{apalike.bst}
\bibliography{technical-arXiv.bib}

\begin{thebibliography}{}

\bibitem[Bai and Saranadasa, 1996]{BS96}
Bai, Z. and Saranadasa, H. (1996).
\newblock Effect of high dimension: by an example of a two sample problem.
\newblock {\em Statistica Sinica}, 6(2):311--329.

\bibitem[Bradley, 2005]{Brad05}
Bradley, R.~C. (2005).
\newblock Basic properties of strong mixing conditions. {A} survey and some
  open questions.
\newblock {\em Probability Surveys}, 2:107--144.
\newblock Update of, and a supplement to, the 1986 original.

\bibitem[Cai et~al., 2014]{CLX14}
Cai, T.~T., Liu, W., and Xia, Y. (2014).
\newblock Two-sample test of high dimensional means under dependence.
\newblock {\em Journal of the Royal Statistical Society. Series B. Statistical
  Methodology}, 76(2):349--372.

\bibitem[Chen and Qin, 2010]{CQ10}
Chen, S.~X. and Qin, Y.-L. (2010).
\newblock A two-sample test for high-dimensional data with applications to
  gene-set testing.
\newblock {\em The Annals of Statistics}, 38(2):808--835.

\bibitem[Choi and Marden, 1997]{CM97}
Choi, K. and Marden, J. (1997).
\newblock An approach to multivariate rank tests in multivariate analysis of
  variance.
\newblock {\em Journal of the American Statistical Association},
  92(440):1581--1590.

\bibitem[Fan and Lin, 1998]{FL98}
Fan, J. and Lin, S.-K. (1998).
\newblock Test of significance when data are curves.
\newblock {\em Journal of the American Statistical Association},
  93(443):1007--1021.

\bibitem[Feng et~al., 2015]{FZWZ15}
Feng, L., Zou, C., Wang, Z., and Zhu, L. (2015).
\newblock Two sample {B}ehrens-{F}isher problem for high-dimensional data.
\newblock {\em Statistica Sinica}.
\newblock To appear.

\bibitem[Gregory et~al., 2014]{GCBL14}
Gregory, K.~B., Carroll, R.~J., Baladandayuthapani, V., and Lahiri, S.~N.
  (2014).
\newblock A two-sample test for equality of means in high dimension.
\newblock {\em Journal of the American Statistical Association}.

\bibitem[Hettmansperger and McKean, 2011]{HM11}
Hettmansperger, T.~P. and McKean, J.~W. (2011).
\newblock {\em Robust nonparametric statistical methods}, volume 119 of {\em
  Monographs on Statistics and Applied Probability}.
\newblock CRC Press, Boca Raton, FL, second edition.

\bibitem[Ibragimov and Linnik, 1971]{IL71}
Ibragimov, I.~A. and Linnik, Y.~V. (1971).
\newblock {\em Independent and stationary sequences of random variables}.
\newblock Wolters-Noordhoff Publishing, Groningen.
\newblock With a supplementary chapter by I. A. Ibragimov and V. V. Petrov,
  Translation from the Russian edited by J. F. C. Kingman.

\bibitem[Kallenberg, 2005]{Kall05}
Kallenberg, O. (2005).
\newblock {\em Probabilistic symmetries and invariance principles}.
\newblock Probability and its Applications (New York). Springer, New York.

\bibitem[Katayama and Kano, 2014]{KK14}
Katayama, S. and Kano, Y. (2014).
\newblock A new test on high-dimensional mean vector without any assumption on
  population covariance matrix.
\newblock {\em Communications in Statistics - Theory and Methods},
  43(24):5290--5304.

\bibitem[Kolmogorov and Rozanov, 1960]{KR60}
Kolmogorov, A.~N. and Rozanov, J.~A. (1960).
\newblock On a strong mixing condition for stationary {G}aussian processes.
\newblock {\em Akademija Nauk SSSR. Teorija Verojatnoste\u\i\ i ee
  Primenenija}, 5:222--227.

\bibitem[Lin and Lu, 1996]{LL96}
Lin, Z. and Lu, C. (1996).
\newblock {\em Limit theory for mixing dependent random variables}, volume 378
  of {\em Mathematics and its Applications}.
\newblock Kluwer Academic Publishers, Dordrecht; Science Press Beijing, New
  York.

\bibitem[Marden, 1999]{Mard99}
Marden, J.~I. (1999).
\newblock Multivariate rank tests.
\newblock In {\em Multivariate analysis, design of experiments, and survey
  sampling}, volume 159 of {\em Statist. Textbooks Monogr.}, pages 401--432.
  Dekker, New York.

\bibitem[M{\"o}tt{\"o}nen and Oja, 1995]{MO95}
M{\"o}tt{\"o}nen, J. and Oja, H. (1995).
\newblock Multivariate spatial sign and rank methods.
\newblock {\em Journal of Nonparametric Statistics}, 5(2):201--213.

\bibitem[M{\"o}tt{\"o}nen et~al., 1997]{MOT97}
M{\"o}tt{\"o}nen, J., Oja, H., and Tienari, J. (1997).
\newblock On the efficiency of multivariate spatial sign and rank tests.
\newblock {\em The Annals of Statistics}, 25(2):542--552.

\bibitem[Oja, 2010]{Oja10}
Oja, H. (2010).
\newblock {\em Multivariate nonparametric methods with {R}}, volume 199 of {\em
  Lecture Notes in Statistics}.
\newblock Springer, New York.
\newblock An approach based on spatial signs and ranks.

\bibitem[Puri and Sen, 1971]{PS71}
Puri, M.~L. and Sen, P.~K. (1971).
\newblock {\em Nonparametric methods in multivariate analysis}.
\newblock John Wiley\thinspace \&\thinspace Sons, Inc., New York-London-Sydney.

\bibitem[Srivastava et~al., 2013]{SKK13}
Srivastava, M.~S., Katayama, S., and Kano, Y. (2013).
\newblock A two sample test in high dimensional data.
\newblock {\em Journal of Multivariate Analysis}, 114:349--358.

\bibitem[Wang et~al., 2015]{WPL15}
Wang, L., Peng, B., and Li, R. (2015).
\newblock A high-dimensional nonparametric multivariate test for mean vector.
\newblock {\em Journal of the American Statistical Association}.

\bibitem[Wei et~al., 2015]{WLWM15}
Wei, S., Lee, C., Wichers, L., and Marron, J. (2015).
\newblock Direction-projection-permutation for high dimensional hypothesis
  tests.
\newblock {\em Journal of Computational and Graphical Statistics}.
\newblock To appear.

\end{thebibliography}
\end{document}